\documentclass[11pt]{amsart}

\usepackage[pagebackref,colorlinks=true]{hyperref}  

\usepackage{amsfonts}
\usepackage{amsmath,amsthm,amssymb,amscd,enumerate,eucal,url,enumitem}
\usepackage{eucal,url,amssymb,verbatim,
enumerate,amscd,
}
\usepackage{times}
\usepackage{latexsym}
\usepackage{lipsum}
\usepackage{enumitem}

\usepackage{amsfonts}
\usepackage{amsmath,amsthm,amssymb,amscd,enumerate,eucal,url,stmaryrd}
\usepackage{mathtools}
\mathtoolsset{showonlyrefs}
\usepackage{latexsym}
\numberwithin{equation}{section}

\usepackage[margin=1in]{geometry}
\numberwithin{equation}{section}
\newtheorem{thrm}{Theorem}[section]
\newtheorem{lemma}[thrm]{Lemma}

\newcommand{\Vol}{\mathop{{\rm Vol}}}
\newcommand{\Ric}{\mathop{{\rm Ric}}}
\newcommand{\id}{\mathop{{\rm id}}}

\newcommand{\Tr}{\mathop{{\rm tr}}}
\newcommand{\Qpf}{\mathop{{\rm {\mathcal{P}}}}}

\newcommand{\lap}{\triangle}

\def\gr{\nabla f}

\begin{document}

\begin{abstract}
We prove an Obata type rigidity result for the first eigenvalue of the sub-Laplacian on a compact seven dimensional quaternionic contact manifold which satisfies a Lichnerowicz-type  bound on its quaternionic contact Ricci curvature and has a non-negative Paneitz  P-function. In particular, under the stated conditions, the lowest possible eigenvalue of the sub-Laplacian is achieved if and only if the manifold is qc-equivalent to the standard 3-Sasakian sphere.
\end{abstract}

\title[The  Obata  first eigenvalue]{The  Obata  first eigenvalue  theorem on a seven dimensional quaternionic contact manifold}
\date{\today }
\author{Abdelrahman Mohamed}
\address[Abdelrahman Mohamed]{Department of Mathematics\\
University of North Alabama\\
Florence, Alabama, 35630-0002\\
}
\email{amohamed2@una.edu}

\author{Dimiter Vassilev}
\address[Dimiter Vassilev]{ Department of Mathematics and Statistics\\
University of New Mexico\\
Albuquerque, New Mexico, 87131-0001\\
}
\email{vassilev@unm.edu}
\maketitle
\tableofcontents


\setcounter{tocdepth}{2}
\section{Introduction}

 \subsection{The qc Obata first eigenvalue theorems in dimension seven}\label{ss:QC Lich and Obata}

 The goal of this article is to prove an Obata type Theorem \ref{t:mainpan} for  seven dimensional compact quaternionic contact (abbr. qc) manifolds assuming a Lichnerowicz-type qc-Ricci lower curvature bound, see Section \ref{ss:Obata hist} for the classical Riemannian result.  We use the connection and torsion tensors of the  Biquard connection defined in dimension seven by Duchemin  \cite{D}. The associated sub-Laplacian is a sub-elliptic operator, hence  on a compact qc manifold it has a  discrete  spectrum and all eigenfunctions are smooth functions. The basic notions, tensors and notations of the relevant  qc geometry are recalled in Section \ref{ss:qc strs}. In the following theorem $S$ and $T^0$ are, respectively, the normalized qc-scalar curvature and  torsion of the Biquard connection $\nabla$. The horizontal bundle $H$ is the kernel of the $\mathbb{R}^3$ valued 1-form $\eta$ defining the qc structure. Throughout the text, we shall use the non-negative sub-Laplacian, $\lap u=-\Tr^g(\nabla^2 u)$.

\begin{thrm}\label{t:mainpan} Let $(M,\eta)$ be a closed compact QC manifold of dimension seven and $g$ be the horizontal metric. Suppose the following qc-Ricci curvature lower bound holds true
\begin{equation}\label{e:Lich}
{\mathcal{L}}(X,X)\overset{def}{=} 2Sg(X,X) + \frac{10}{3}T^0(X,X) \geq 4g(X,X), \qquad X\in\Gamma(H),
\end{equation}
and the $P$-function of any eigenfunction associated to the first non-zero eigenvalue of the sub-Laplacian is non-negative.  If the (lowest) eigenvalue of the sub-Laplacian is $4$, then $(M,\eta)$ is qc-equivalent to the standard $3$-Sasakian sphere.
\end{thrm}

Following \cite{IPV2}, but working only in dimension seven, for a fixed smooth function $f$ on a seven dimensional qc manifold we define a one form $ P_f$ on $M$, which we call the $P-$form of $f$, by the following equation
\begin{equation}\label{e:Pfn}
P_f(X) =\sum_{b=1}^{4}\nabla ^{3}f(X,e_{b},e_{b})+\sum_{t=1}^{3}\sum_{b=1}^{4}\nabla
^{3}f(I_{t}X,e_{b},I_{t}e_{b})
-4Sdf(X)+4T^{0}(X,\nabla f).
\end{equation}
The $P-$function of $f$ is the function $P_f(\nabla f)$, which is called  non-negative if
\begin{equation}\label{e:Pfn positivity}
-\int_M P_f(\nabla f)\, Vol_{\eta}\geq 0,
\end{equation}
where the volume form is ${\Vol}_\eta = \eta_1 \wedge \eta_2 \wedge \eta_3 \wedge {\Omega}$ with $\Omega=\sum_{s=1}^3\omega_s\wedge\omega_s$, the so called fundamental 4-form of the qc-structure.   We note that it was proven in \cite[Proposition 3.4]{IPV2} that on a locally 3-Sasakian manifold the $P-$function of any {eigenfunction} of the sub-Laplacian is non-negative.

In order to recognize the theorem as an Obata type result we recall the Lichnerowicz estimate of the first eigenvalue established in \cite{IPV1} in dimensions greater than seven and in \cite{IPV2}  in the seven
dimensional case.  Let $(M,\eta)$ be a compact QC manifold of dimension $4n+3$,
\[
\alpha_n=\frac {2(2n+3)}{2n+1}, \qquad \beta_n=\frac {4(2n-1)(n+2)}{(2n+1)(n-1)}, \quad\text{with}\quad \beta_1=0,
\]
and $T^0$ and $U$ be the components of the torsion tensor, which up to multiplicative constants, are the traceless $Sp(n)Sp(1)$ invariant components of the qc-Ricci tensor  \cite{IMV}, see Section \ref{ss:qc strs}.
Suppose that  for any horizontal vector field $ X\in H$ the following Lichnerowcz-type bound holds true,
$$\mathcal{L}(X,X)\overset{def}{=}2 Sg(X,X)+\alpha_n T^0(X,X) +\beta_n U(X,X)\geq 4g(X,X).$$
In the case $n=1$ assume, in addition,  the positivity of
the $P$-function  of any eigenfunction, cf. \eqref{e:Pfn positivity}. Then, any eigenvalue $\lambda$ of the sub-Laplacian $\triangle$ satisfies the inequality $\lambda \ge 4n.$  A small calculation shows, see \cite{IMV2},
\cite{ACD} and also \cite[Theorem 4.1]{IV15},  that the  3-Sasakian sphere achieves equality in the bottom of spectrum inequality. In fact, on the 3-Sasakian sphere the eigenspace of the sub-Laplacian  with eigenvalue $4n$ is given by the restrictions to the sphere of all linear functions in Euclidean space.

The  rigidity result when the dimension of the qc manifold is at least eleven, i.e.,  the Obata type theorem characterizing the 3-Sasakian sphere as the only case in which the lowest possible eigenvalue is achieved was proven in \cite{IPV3}. In fact, \cite{IPV3} established a general result valid on any complete with respect to the associated Riemannian metric qc manifold, characterizing the 3-Sasakian sphere of dimension at least eleven through the existence of an eigenfunction whose Hessian with respect to the Biquard connection \cite{Biq1} is in the space generated by the metric and fundamental 2-forms of the quaternionic contact structure. In the seven dimensional case, the claim of Theorem \ref{t:mainpan} was proven in \cite[Corrolary 1.2]{IPV2} assuming in addition to the stated hypotheses  that the torsion tensor $T^0=0$ and that the qc-scalar curvature is constant. The latter condition was shown to be redundant in \cite{IMV3} since it follows from the vanishing of the torsion. The general qc Obata result in dimension seven remained open which motivated the current article where we prove that the vanishing of the torsion is implied from the stated conditions in Theorem \ref{t:mainpan}.

In view of the results of \cite{IPV2} and \cite{IMV3}  mentioned in the preceding paragraph,  the proof of Theorem \ref{t:mainpan} amounts to showing that the torsion tensor vanishes. The seven dimensional case  result differs from the proof of the higher dimensional version \cite{IPV3} in several aspects. First,  we begin by using  the compactness of the manifold to obtain the vanishing of the P-form of any eigenfunction $f$ with eigenvalue achieving the lowest possible value 4. In fact, we interpret  the  Lichnerowicz' condition as a non-negativity of a certain quadratic form $\Qpf$, cf.  \eqref{e:Qpf}, related to the $P$-form \eqref{e:Pfn} of $f$, which together with the assumed positivity of the $P-$function of $f$ shows that the horizontal gradient of $f$ is in the kernel of $\Qpf$.   This allows one to see fairly easily that certain components of the torsion tensor vanish. The vanishing of  $\Qpf$ is equivalent to the vanishing of the torsion tensor. This brings about another substantial difference between the $n=1$ and $n>1$ cases. In the higher dimensional case  \cite[Lemma 3.8]{IPV3}, one sees that the two $Sp(n)Sp(1)$ invariant torsion tensors $T^0$ and $U$, which determine the full torsion and the traceless part of the qc-Ricci tensor, can be expressed in terms of the function $U(\nabla f,\nabla f)$ and the  horizontal gradient of $f$ as follows
\begin{equation}\label{t000}
|\nabla f|^{4}T^{0}(X,Y)=-\frac{2n}{n-1}U(\nabla f,\nabla f)\Big[%
3df(X)df(Y)-\sum_{s=1}^{3}df(I_{s}X)df(I_{s}Y)\Big],
\end{equation}%
\begin{equation}\label{u000}
|\nabla f|^{4}U(X,Y)=-\frac{1}{n-1}U(\nabla f,\nabla f)\Big[|\nabla
f|^{2}g(X,Y)-n\Big(df(X)df(Y)+\sum_{s=1}^{3}df(I_{s}X)df(I_{s}Y)\Big)\Big].
\end{equation}
  Thus, in the higher dimensional case, the crux of the matter  in showing that the torsion vanishes is the proof that $U(\gr,\gr)=0$, which is achieved with the help of the Ricci identities and the contracted Bianchi second identity. In dimension seven, the tensor $U$ vanishes trivially and there is no substitute for \eqref{t000}.  In addition, the contracted Bianchi second identity brings no constraints since it holds trivially. We remark that this is one of the reasons for the difficulties in showing that a  qc-Einstein seven dimensional manifold is of constant scalar curvature, which was shown in \cite{IMV3}. In our case, one of the crucial steps turns out to be also that the qc scalar curvature is constant, see \ref{ss:S is const}.

We should mention that in the proofs of the sub-Riemannian versions of the rigidity result eventually one resorts to using the Riemannian Obata theorem to obtain an isometric equivalence to the round Euclidean sphere after which one argues that the equivalence extends to the additional quaternionic contact or CR structures in the case of the CR version. A more detailed review and references  of the corresponding problems, including the geometric interpretation of the functions realizing the equality cases and the  role played by the rigidity result in the uniqueness of qc and CR Yamabe metrics within a fixed qc or, in the CR case, pseudohermitian conformal class can be found in \cite{IV15} and \cite{IMV15a}. In the next few paragraphs we give only a brief background.

\subsection{Riemannian and K\"ahler cases}\label{ss:Obata hist}
  Lichnerowicz \cite{Li}  showed
that on a compact Riemannian manifold $(M,g)$ of dimension $n$ for which the
Ricci curvature  satisfies $
\Ric(X,X)\geq (n-1)g(X,X) $  the first positive eigenvalue $\lambda_1$ of the (positive) Laplace operator $\lap f=-\Tr^g\nabla^2f$
satisfies the inequality $%
\lambda_1\geq n$ where $\nabla$ is the Levi-Civita connection of $g$.
Subsequently, Obata \cite{O3} proved that equality is
achieved if and only if the Riemannian manifold is isometric to the round unit sphere with eigenfunctions given by the spherical harmonics of degree one. Obata's rigidity theorem  was preceded by several results where the case of equality was characterized under the additional assumption that $g$ is Einstein \cite{YaNa59} or has constant scalar curvature \cite{IshTa59}, see also \cite[Proposition 2.4]{IV15} where it is shown that the metric is Einstein automatically.

In the K\"ahler case,  Lichnerowicz showed an
improvement of the first eigenvalue inequality by showing that if $ M$ is a closed K\"ahler manifold with $\Ric \ge k$, then the first non-zero eigenvalue satisfies $\lambda_1 \ge 2k$. Furthermore, equality implies that the gradient field of any eigenfunction for $\lambda_1$ is a (non-trivial) real holomorphic vector field.

\subsection{The CR case}
From the sub-ellipticity of the sub-Laplacian on a strictly pseudoconvex CR manifold it follows that on a compact manifold its spectrum is
discrete. A CR analogue of the Lichnerowicz theorem was found by
Greenleaf \cite{Gr} for dimensions $2n+1>5$,  while the corresponding results for  $n=2$ and  $n=1$ were achieved later in \cite{LL}  and
\cite{Chi06}, respectively. In this case, the standard Sasakian unit sphere has first eigenvalue
equal to 2n with eigenspace spanned by the restrictions of
all linear functions to the sphere. Up to scaling it provides the equality case in the bottom of the spectrum estimate.   As far as the Obata type result is concerned, the most general result valid on a complete CR manifold was proven in \cite{IVO} under the assumption of a  divergence-free  pseudohermitian torsion. In the compact case, \cite{LW,LW1,IV3} proved the Obata type theorem on  a compact strictly pseudoconvex pseudohemitian manifold which satisfies the Lichnerowicz-type bound. In dimension three, in all of the above results it is assumed that  the CR-Paneitz operator is non-negative.

In \cite{ChChY12} and \cite{LSW15}, it was proven an analogue of the Lichnerowicz eigenvalue estimate for the Kohn Laplacian assuming that the CR Paneitz operator is non-negative in dimension three. Furthermore, \cite{LSW15} and \cite{CY2} established the corresponding rigidity results. We mention that a big part of  the difficulty in  the three dimensional case stems from the possible difference between the kernel of the CR Paneitz operator and the space of CR pluriharmonic functions.

\subsection{Conventions}\label{Conv}
\begin{enumerate}[label=\alph*)]
\item We shall use $X,Y,Z$ to denote horizontal vector fields, i.e., $X,Y,Z\in\Gamma(H)$.
\item The triple $(i,j,k)$ denotes any cyclic permutation of $(1, 2, 3)$.
\item The sum $\sum_{(ijk)}$ means the cyclic sum. For example,
$$T_{11}^2 + T_{22}^2 + T_{33}^2 + 2T_{12}^2 + 2T_{23}^2 + 2T_{31}^2
=\sum_{(i\,j\,k)} [T_{ii}^2 + 2T_{ij}^2]=\sum_{(i\,j\,k)} [T_{ii}^2 + 2T_{jk}^2].$$
\item In general,  $s$ will be any number from the set $\{1, 2, 3\}$.
\item Greek indices $\alpha$, $\beta$ etc., will be in the set $\{0,1,2,3\}$.
\item $\{{e_\gamma}\}_{{\gamma=0}}^{3}$ denotes a local orthonormal basis of the horizontal space $H$.
\item The summation convention over repeated vectors from the basis  $\{{e_\gamma}\}_{{a=0}}^{3}$ will be used. For example, for a $(0,4)$-tensor $P$, the formula $k = P({e_\beta}, {e_\gamma}, {e_\gamma}, {e_\beta})$ means $k = \sum_{\beta,\, \gamma =0}^{3}P({e_\beta}, {e_\gamma}, {e_\gamma}, {e_\beta})$.
\item The following divergences for $\alpha=0,1,2,3$ will be used
\begin{equation}\label{e:def of div}
\nabla^*T^0(X) \overset{def}{=} \nabla T^0({e_\gamma}, {e_\gamma}, X), \quad
\nabla_\alpha^* T^0(X) \overset{def}{=} \nabla T^0({e_\gamma}, I_\alpha {e_\gamma}, X)
\end{equation}
where $I_0$ is the identity on the horizontal space $H$.
\end{enumerate}

\textbf{Acknowledgements} The second author would like to thank  Stefan Ivanov and Alexander Petkov  for many discussions in our earlier joint attempt to prove the result of this paper, which lead to Lemma \ref{e:old}. D.V was partially supported by ARPA-E contract number DE-AR0001202. A.M was partially supported by the Efroymson Foundation at UNM.

\section{Proof of the main Theorem}
\subsection{Quaternionic contact structures and the Biquard connection.}\label{ss:qc strs}

In this section we set the notation and briefly review the necessary notions of (seven dimensional) quaternionic contact geometry \cite{Biq1}, \cite{D}, \cite{IMV} and \cite{IV2}. We will consider a seven dimensional integrable quaternionic contact (qc) manifold $(M,\eta,g,\mathbb{Q})$, see \cite{D}. Thus, we have a seven dimensional manifold $M$ equipped with an oriented  codimension three sub-bundle $H$ of the tangent bundle, such that $H$ is locally given as the kernel of a $1$-form $\eta = (\eta_1, \eta_2,\eta_3)$ with values in $\mathbb{R}^3$. In addition $H$ has an $Sp(1)Sp(1)$ structure, that is, it is equipped with a Riemannian metric $g$ and a rank-three bundle $\mathbb{Q}$ consisting of endomorphisms of $H$ locally generated by three almost-complex structures $I_1, I_2, I_3$ on $H$ satisfying the identities of the imaginary unit quaternions, $I_1 I_2 = - I_2 I_1 = I_3$ and $I_1I_2I_3 = - \id_H$, which are Hermitian compatible with the metric $g(I_s\cdot, I_s \cdot) = g(\cdot, \cdot)$ and the following compatibility condition holds,  $$2g(I_s X, Y) = d\eta_s(X,Y), \qquad X,Y\in H.$$ In addition, we assume that the qc-structure is integrable, i.e.,  there exist three smooth vector fields $\{\xi_1, \xi_2, \xi_3\}$, called Reeb vector fields, which satisfy the identities
\begin{equation}\label{etaxi}
\eta_s(\xi_k) = \delta_{sk}, \qquad (\xi_s \lrcorner d\eta_s)|_H = 0, \qquad (\xi_s \lrcorner d\eta_k)|_H = -(\xi_k \lrcorner d\eta_s)|_H.
\end{equation}
The Reeb vector fields define the so called vertical space, which will be denoted by $V$.  Using the triple of Reeb vector fields we extend the ''horizontal'' metric $g$ on $H$ to a metric $h$ on $TM$ by requiring that the Reeb vector fields are an orthonormal frame for the vertical space $V$, $h(\xi_s, \xi_t) = \delta_{st}$. Thus, with a slight abuse of notation, we have
\begin{equation}\label{e:h metric}
h=g+\sum_{s=1}^3 \eta_s^2.
\end{equation}
The Reimannian metric $h$ as well as the canonical connection do not depend of the action of $SO(3)$ on $V$, but both change if $\eta$ is multiplied by a conformal factor \cite{IMV}.
The fundamental $2$-forms ${\omega_s}$ of the quaternionic structure $\mathbb{Q}$ are defined by
\begin{equation}\label{fundforms}
2{\omega_s}|_H = d\eta_s|_H, \qquad \xi\lrcorner{\omega_s} = 0, \,\, \xi\in V.
\end{equation}

We will denote by $\nabla$ the ''canonical'' Biquard connection, which in the considered seven dimensional case was  defined by Duchemin \cite{D}, but we will use the conventions adopted by Biquard  \cite{Biq1}. The two connections differ only in the derivatives of type $\nabla_{\xi}\xi'$ for $\xi, \xi'\in V$. The torsion of the connection $\nabla$ will be denoted by $T$. In particular we have that
 $\nabla$ preserves the decomposition $H\oplus V$ and the $Sp(1)Sp(1)$ structure on $H$, i.e., $\nabla g = 0$, $\nabla\sigma \in\Gamma(\mathbb{Q})$ for a section $\sigma\in\Gamma(\mathbb{Q})$, and its torsion on $H$ is given by $T(X,Y) = -[X,Y]|_V$. Furthermore, for  a vertical vector $\xi\in V$, the endomorphism $T(\xi,\cdot)|_H$ of $H$ lies in $(sp(1)\oplus sp(1))^\perp \subset gl(4)$. Finally, the connection preserves the bundle  of self-dual 2-forms on $H$ with respect to the conformal class defined by $g$ and we have
\begin{equation}\label{e:Biq}
\nabla I_i = -\alpha_j \otimes I_k + \alpha_k \otimes I_j, \qquad \nabla \xi_i = -\alpha_j \otimes \xi_k + \alpha_k \otimes \xi_j.
\end{equation}
The vanishing of  the $sp(1)$-connection $1$-forms on $H$ implies the vanishing of the torsion endomorphism of the canonical connection, see \cite{IMV}.  Clearly, the canonical connection preserves the Riemannian metric on $TM$, i.e., $\nabla h = 0$.

Due to \eqref{fundforms}, the torsion restricted to $H$ has the form
\begin{equation}
T(X,Y) = -[X,Y]|_V =2\omega_1(X,Y)\xi_1 + 2\omega_2(X,Y)\xi_2 + 2\omega_3(X,Y)\xi_3.
\end{equation}
On the other hand, in the seven dimensional case,  the endomorphism $T_\xi=T(\xi,.) \in (sp(1) \oplus sp(1))^\perp$ is symmetric and completely trace-free, $$\Tr T_\xi = \Tr T_\xi \circ I_s = 0,\qquad T_\xi =T_{\xi}^0,$$ where following tradition, and to be consistent with the higher dimensional case, we denote with $T_{\xi}^0$ the symmetric part of the endomorphism $T_\xi$. The  following $Sp(1)Sp(1)$-invariant trace-free symmetric $2$-tensor  was introduced in \cite{IMV},
\begin{equation}\label{e:T0}
T^0(X,Y) = g((T_{\xi_1}^0I_1 +T_{\xi_2}^0I_2 + T_{\xi_3}^0I_3)X,Y), \qquad X,Y\in\Gamma(H).
\end{equation}
The tensor $T^0$ belongs to the $-1$ eigenspace   of the Casimir operator $\Upsilon = I_1 \otimes I_1 + I_2 \otimes I_2 + I_3 \otimes I_3$ which satisfies $(\Upsilon -3I)(\Upsilon +I)$, i.e.,
\begin{equation}\label{e:Tss}
T^0(X,Y) + T^0(I_1 X, I_1 Y) + T^0(I_2 X, I_2 Y) + T^0(I_3 X, I_3 Y)= 0.
\end{equation}
Furthermore, in dimension seven, the torsion is determined by $T^0$ through the following formula,  \cite[Proposition 2.3]{IV} or \cite[Lemma 4.2.6]{IV2},
\begin{equation}\label{TorT0}
T(\xi_s, I_s X, Y) = T^0(\xi_s, I_s X, Y)  = \frac{1}{4}[T^0(X,Y) - T^0(I_s X, I_s Y)] .
\end{equation}
Up to a multiplicative constant, $T^0$ gives the trace-free part of the qc-Ricci tensor defined below, see \eqref{e:2.10}.

We let $R = [\nabla, \nabla] - \nabla_{[\,\, ,\,\,]}$ be the curvature tensor of $\nabla$. We shall also use $R$ and $T$ to denote, correspondingly, the curvature tensor of type $(0,4)$ and the torsion tensor of type $(0,3)$, \[
R(A,B,C,D)=h(R(A,B)C,D),\quad T(A,B,C)= h(T(A,B),C),\quad A,B,C,D\in\Gamma(TM).
\]
The \emph{qc-Ricci} tensor $\Ric$, the \emph{normalized qc-scalar curvature} $S$, the \emph{qc-Ricci} $2$-forms $\rho_s$, and the \emph{qc-Ricci}-type tensors $\zeta_s$ are defined by
\begin{equation}\label{qc-Ric}
\begin{aligned}
\Ric(A,B) = R({e_\beta}, A, B, {e_\beta}), \qquad S = \frac {1}{24}R({e_\beta}, {e_\gamma}, {e_\gamma}, {e_\beta}), \\
\rho_s(A,B) = \frac{1}{4} R(A,B, {e_\gamma}, I_s {e_\gamma}), \qquad \zeta_s(A,B) = \frac{1}{4} R({e_\gamma}, A, B, I_s {e_\gamma}).
\end{aligned}
\end{equation}

The qc-Ricci tensor can be expressed in terms of the torsion of the Biquard connection \cite{IMV}, see also \cite{IMV1, IV}. Furthermore, we have the following identities valid when $n=1$, see \cite[Theorem $1.3$, Theorem $3.12$, Corollary $3.14$, Proposition $4.3$ and Proposition $4.4$]{IMV},
\begin{equation}\label{e:2.10}
\begin{aligned}
&\Ric(X,Y) = 4T^0(X,Y)  + 6Sg(X,Y), \\
&\zeta_s(X,I_sY) = \frac{3}{4}T^0(X,Y) + \frac{1}{4}T^0(I_s X, I_s Y) + \frac{S}{2}g(X,Y), \\
&T(\xi_i, \xi_j) = -S\xi_k - [\xi_i, \xi_j]|_H,\\
&g(T(\xi_i, \xi_j), X) = -\rho_k(I_i X, \xi_i) = - \rho_k(I_j X, \xi_j) = -h([\xi_i, \xi_j],X).
\end{aligned}
\end{equation}

We recall that a qc-structure is called \emph{qc-Einstein} if the horizontal qc-Ricci tensor is a scalar multiple of the metric, $$\Ric(X,Y) = 6Sg(X,Y).$$  From the above formulas, it follows that the structure is  qc-Einstein if and only if $T^0=0$, which is equivalent to the vanishing of  the torsion endomorphism. In this case the normalized qc-scalar curvature $S$ is constant and the vertical distribution $V$ is integrable  \cite{IMV3}, see \cite{IMV} for $n>1$. If $S >0$ then the qc-manifold is locally $3$-Sasakian \cite{IMV}, see \cite{Fund4} for the negative qc-scalar curvature case.

 The (horizontal) divergence of a horizontal vector field/$1$-form $\sigma\in\Lambda^1(H)$, defined by $$\nabla^*\sigma = -{\Tr}|_H(\nabla\sigma) = -\nabla\sigma({e_\gamma}, {e_\gamma}).$$ We have the following "integration by parts" formula on a compact $M$ formula \cite{IMV}, see also \cite{YamP},
\begin{equation}\label{e:divthrm}
\int_M (\nabla^* \sigma) {\Vol}_\eta = 0.
\end{equation}
The sub-Laplacian $\Delta f$, and the norm of the horizontal gradient $\nabla f$, of a smooth function $f$ on $M$ are defined, respectively, by
\[\Delta f = -{\Tr}^g_H(\nabla^2 f) = \nabla^*df = -\nabla^2 f({e_\gamma}, {e_\gamma}), \qquad |\nabla f|^2 = df({e_\gamma})df({e_\gamma}).\]
The function $f$ is an eigenfunction  of the sub-Laplacian with eigenvalue $\lambda$ if for some constant $\lambda$ we have
\begin{equation}
\Delta f = \lambda f.
\end{equation}
On the other hand, the Ricci identity implies
\begin{equation}\label{vderiv}
g(\nabla^2 f, {\omega_s}) = \nabla^2 f({e_\gamma}, I_s {e_\gamma}) = - 4 df(\xi_s)=-4f_s,
\end{equation}
where we set
\begin{equation}\label{e:vert deriv}
f_s \overset{def}{=} df(\xi_s).
\end{equation}

\subsection{Some basic identities}
We are ready to begin the proof of Theorem \ref{t:mainpan}. We shall assume throughout  all of the assumptions of Theorem \ref{t:mainpan}. The function $f$ will denote an eigenfunction achieving the lowest possible eigenvalue   $\lambda = 4$ of the sub-Laplacian.

It was shown in \cite[Remark 4.1]{IPV2} that, with the made assumptions,  the horizontal Hessian of $f$ is given by
\begin{equation}\label{e:Hessf}
\nabla^2 f(Y,X) = -fg(Y,X) - \sum_{s=1}^3 {f_s} {\omega_s}(Y,X), \qquad Y,X\in\Gamma(H).
\end{equation}
and
\begin{equation}  \label{e:extr}
 {\mathcal{L}}(\nabla f,\nabla f) - 4|\nabla f|^2=2\left[(S-2)|\nabla
f|^2+\frac {5}{3}T^{0}(\nabla f,\nabla f)\right]=0, \quad \int_M P_f(\nabla f)\, Vol_{\eta}=0,
\end{equation}
recalling the notation $f_s=df(\xi_s)$ set in \eqref{e:vert deriv}.
We note that the compactness of $M$ was essential in order to obtain the above identities by integrating the qc Bochner formula. Furthermore, if $f$ satisfies \eqref{e:Hessf}, then differentiating the Hessian equation we obtain the identity, \cite[Lemma 3.1]{IPV3},
\begin{equation}\label{e:Nab3f}
\nabla^3 f(A,Y,X) = -df(A)g(Y,X) - \sum_{s=1}^3 \nabla^2 f(A,\xi_s)  {\omega_s}(Y,X), \quad A\in\Gamma(TM), \,\,Y,X \in\Gamma(H).
\end{equation}
  We note that \cite{IPV3} assumes  $n>1$, but the cited lemma and its proof do not make use of this assumption. In addition, the argument leading to \cite[(3.8)]{IPV3} is valid in the case $n=1$ as well, i.e., we have the following identity
 \begin{equation}\label{e:1}
\sum_{s=1}^3 \nabla^2 f(I_s X, \xi_s) = (1-2S)df(X) - \frac{2}{3} T^0(X,\nabla f),
\end{equation}
which follows from the Ricci identity $\nabla^2 f(X,\xi_s) - \nabla^2 f(\xi_s, X) = T(\xi_s, X, \nabla f)$ applied to the left-hand side of   \cite[(3.8)]{IPV3}, noting that the Ricci idenity and   \eqref{TorT0} give
\begin{equation}\label{e:Ricci and TorT0}
\nabla^2 f(X, \xi_s) - \nabla^2 f(\xi_s, X) = - \frac{1}{4}[T^0(I_s X, \nabla f) + T^0(X, I_s \nabla f)].
\end{equation}

It will be convenient to define the quadratic symmetric (0,2)-tensor
$\Qpf$ by
\begin{equation}\label{e:Qpf}
\Qpf(X,Y) \overset{def}{=} 2\left[{\mathcal{L}}(X,Y) - 4g(X,Y)\right]=4\left[(S-2)g(X,Y) + \frac{5}{3}T^0(X,Y) \right].
\end{equation}
The Lichnerowicz-type bound  \eqref{e:Lich}   implies that  $\Qpf$ is non-negative $\Qpf(X,X)\geq 0$, hence, taking into account that $T^0$ is a traceless tensor, we have $S\geq 2$, while by \eqref{e:extr} we have
\[
\Qpf(\nabla f,\nabla f)=0.
\]

\begin{lemma}\label{lemPf}
The  $P$-form  of $f$ is $\Qpf(X,\nabla f)$, i.e.,
\begin{equation}\label{e:Pformf}
P_f(X) =\Qpf(X,\nabla f)= 4(S-2)df(X) + \frac{20}{3}T^0(X,\nabla f).
\end{equation}
Furthermore, $\Qpf(X,\nabla f)=0$, hence
\begin{equation}\label{e:T0Xf}
T^0(X, \nabla f) = -\frac{3}{5}(S-2)df(X)
\end{equation}
\end{lemma}

\begin{proof}
Taking the appropriate traces in \eqref{e:Nab3f} we obtain
\begin{align*}
\nabla^3 f(X,{e_\gamma},{e_\gamma}) & =-4df(X),\\
 \sum_{s=1}^3\nabla^3 f(I_sX, {e_\gamma}, I_s{e_\gamma})& =-4\sum_{s=1}^3\nabla^2f(I_sX,\xi_s)=-4(1-2S)df(X)+\frac 83T^0(X,\nabla f).
\end{align*}
A substitution of the above two identities  in the definition \eqref{e:Pfn} of the $P$-form of $f$ gives \eqref{e:Pformf}. The non-negativity of $\Qpf$, Cauchy-Schwarz' inequality and $\Qpf(\nabla f,\nabla f)=0$ imply $\Qpf(X,\nabla f)=0$.
\end{proof}

An immediate consequence of the above lemma are the following identities,
\begin{equation}\label{e:T0i}
T^0(I_s \nabla f, \nabla f) =0, \qquad s=1,2,3.
\end{equation}
In addition,  the covariant derivative of \eqref{e:T0Xf} along a horizontal vector $Y$ and the Hessian equation \eqref{e:Hessf} yield the following equation
\begin{multline}\label{e:nablaT0}
\nabla T^0(Y,X,\nabla f)  = -\frac{3}{5}dS(Y)df(X) + f\left(T^0(Y,X) + \frac{3}{5}(S-2)g(Y,X) \right) \\
+\sum_{s=1}^3{f_s}\left(T^0(I_sY,X) + \frac{3}{5}(S-2)\omega_s(Y,X) \right).
\end{multline}

\subsubsection{} Next we will obtain formulas for the individual terms in the sum \eqref{e:1} and  for the derivative of the obtained identities. This will be achieved by computing the horizontal qc-Ricci tensor $\zeta_s$ in two different ways.

\begin{lemma}\label{e:old}
The following identities hold true 
\begin{equation}\label{e:nabla2f1}
\nabla^2 f(X, \xi_i) = \frac{1}{5}(1+2S)df(I_i X)-\frac{2}{3}T^0(X, I_i \nabla f)
\end{equation}
and
\begin{multline}\label{e:nabla3f1}
\nabla^3 f(Y,X,\xi_i) =  \frac{1}{5}(1+2S)\left[ f\omega_i(Y,X) -{f_i} g(Y,X)+{f_j}\omega_k(Y,X)-{f_k}\omega_j(Y,X)\right]\\
 +\frac{2}{3}\left[ fT^0(X, I_i Y)- {f_i}T^0(X,  Y)+ {f_j}T^0(X, I_k Y)-{f_k}T^0(X, I_j Y) \right]\\
+\frac{2}{5}dS(Y)df(I_i X)-\frac{2}{3}\nabla T^0(Y, X, I_i \nabla f) .
\end{multline}
\end{lemma}

\begin{proof}
First, we compute $\zeta_i(I_i X, \nabla f)$ using its definition \eqref{qc-Ric} and then apply the following third order Ricci identity
\[
R(X,Y,\nabla f,Z) =\nabla^3 f(X,Y,Z) - \nabla^3 f(Y,X,Z)+ 2 \sum_{s=1}^3\nabla^2 f(\xi_s, Z) {\omega_s}(X,Y),
\]
which give for a $i$ fixed the formula
\begin{multline}
4\zeta_i(I_i X, \nabla f) = R({e_\gamma}, I_iX,\nabla f, I_i {e_\gamma})
=\nabla^3 f({e_\gamma},I_i X,I_i {e_\gamma}) - \nabla^3 f(I_i X,{e_\gamma},I_i {e_\gamma})\\
 + 2 \sum_{s=1}^3\nabla^2 f(\xi_s, I_i {e_\gamma}) {\omega_s}({e_\gamma},I_i X).
\end{multline}
An application of \eqref{e:Nab3f} and \eqref{e:1} to the above equation brings us to
\begin{multline}\label{e:0.2}
4\zeta_i(I_i X, \nabla f)
= - df(X) - \sum_{s=1}^3 \nabla^2 f(e_\gamma, \xi_s)\omega_s(I_i X, I_i e_\gamma) + \sum_{s=1}^3 \nabla^2 f(I_i X, \xi_s)\omega_s(e_\gamma, I_i e_\gamma) \\
+ 2 \sum_{s=1}^3 \nabla^2 f(\xi_s,I_i e_\gamma)\omega_s(e_\gamma, I_i X)
=  - df(X) - \left[\nabla^2 f(I_i X, \xi_i) - \nabla^2 f(I_j X, \xi_j) - \nabla^2 f(I_k X, \xi_k)\right] \\
+ 4\nabla^2 f(I_i X, \xi_i)
+ 2[\nabla^2 f(\xi_i, I_i X) - \nabla^2 f(\xi_j, I_j X) - \nabla^2 f(\xi_k, I_k X)].
\end{multline}
Invoking \eqref{e:Ricci and TorT0} to re-write the last bracket in \eqref{e:0.2} we come to
\begin{multline}
4\zeta_i(I_i X, \nabla f) = - df(X) - \left[\nabla^2 f(I_i X, \xi_i) - \nabla^2 f(I_j X, \xi_j) - \nabla^2 f(I_k X, \xi_k)\right] + 4\nabla^2 f(I_i X, \xi_i) \\+ 2[\nabla^2 f(I_i X, \xi_i)-\nabla^2 f(I_j X, \xi_j) - \nabla^2 f(I_k X, \xi_k)] -\frac{1}{2}[T^0(X,\nabla f) - T^0(I_i X, I_i \nabla f)] \\+\frac{1}{2}[T^0(X,\nabla f) - T^0(I_j X, I_j \nabla f)] + \frac{1}{2}[T^0(X,\nabla f) - T^0(I_k X, I_k \nabla f)]  \\= - df(X) + 6\nabla^2 f(I_i X, \xi_i) - \sum_{s=1}^3 \nabla^2 f(I_s X, \xi_s) + T^0(X,\nabla f) + T^0(I_i X, I_i \nabla f)\\- \frac{1}{2}\left[T^0(X,\nabla f) + \sum_{s=1}^3 T^0(I_s X, I_s \nabla f)\right].
\end{multline}
Finally, using \eqref{e:Tss} and \eqref{e:1} the last equation takes the form
\begin{equation}\label{e:0.3} 4\zeta_i(I_i X, \nabla f) = 2(S-1)df(X) + 6\nabla^2 f(I_i X, \xi_i) + \frac{5}{3}T^0(X,\nabla f) + T^0(I_i X, I_i \nabla f).\end{equation}

On the other hand,  from \eqref{e:2.10} we have the following formula for $\zeta_i(I_i X, \nabla f)$
\begin{equation}\label{e:0.4}
4\zeta_i(I_i X, \nabla f) = -3T^0(I_i X, I_i \nabla f) - T^0(X,\nabla f) -2Sdf(X).
\end{equation}

From \eqref{e:0.4} and \eqref{e:0.3} we obtain
\begin{equation}\label{e:eq2}
3\nabla^2 f(X, \xi_i) = (2S-1)df(I_i X) +\frac{4}{3}T^0(I_i X,\nabla f) - 2 T^0( X, I_i \nabla f).
\end{equation}
Finally, a substitution of \eqref{e:T0Xf} into \eqref{e:eq2} gives \eqref{e:nabla2f1}.

The second identity in the lemma is obtained by taking the covariant derivative of \eqref{e:nabla2f1} along $Y$, noting that from \eqref{e:Biq}  the terms involving the connection $1$-forms coming from the covariant derivatives of $ I_i$ and $ \xi_i$ cancel, which gives
\begin{multline*}
\nabla^3 f(Y,X,\xi_i) = \frac{2}{5}dS(Y)df(I_i X) -\frac{2}{3}\nabla T^0(Y, X, I_i \nabla f) + \frac{1}{5}(1+2S)\nabla^2 f(Y, I_i X) \\
 +\frac{2}{3} fT^0(X, I_i Y) + \frac{2}{3}\sum_{s=1}^3 {f_s}T^0(X, I_i I_s Y),
\end{multline*}
Finally, using  the Hessian equation \eqref{e:Hessf} in the above formula gives \eqref{e:nabla3f1}.
\end{proof}

Some of the above identities  can be viewed as versions of formulas found in \cite{IPV3} that hold when $n=1$. Other than \eqref{e:Hessf} and \eqref{e:Nab3f} coming directly from \cite{IPV3} as stated above, \eqref{e:1} can be obtained from \cite[$(3.8)$]{IPV3} by  setting $U=0$, $n=1$, and then applying a Ricci identity.

On the other hand, identity \eqref{e:T0Xf} can be \emph{formally} obtained from \cite[$(3.5)$]{IPV3} by setting $U=0$ and $n=1$. When $n>1$, the proof of \cite[Lemma $3.2$]{IPV3} shows that \cite[$(3.5)$]{IPV3} is found  by subtracting \cite[$(3.7)$]{IPV3} from \cite[$(3.8)$]{IPV3}. However, if $n=1$ then \cite[$(3.7)$]{IPV3} is identical to \cite[$(3.8)$]{IPV3} and therefore we cannot obtain  \cite[$(3.5)$]{IPV3} when $n=1$ following \cite{IPV3}. In our case, in order to prove \eqref{e:T0Xf} we used  the compactness of $M$ to show that the $P$-form of $f$ vanishes instead. However, once \eqref{e:T0Xf} is known, substituting it into \eqref{e:1}, and then using a Ricci identity, will yield \cite[$(3.9)$]{IPV3}. In addition, \eqref{e:T0i} implies that \cite[$(3.6)$]{IPV3} continues to hold when $n=1$.
Finally,  \eqref{e:nabla2f1} and \eqref{e:nabla3f1} correspond to  \cite[Lemma 3.3]{IPV3}  and \cite[(3.18)]{IPV3}, respectively, but now they hold  in the case $n=1$.

\subsubsection{A key identity.}
Since the canonical connection preserves the type of a tensor and $T^0$ is symmetric, we can compute terms of the form $\nabla T^0(Y, \nabla f, I_i \nabla f)$ by finding $\nabla T^0(Y, I_i \nabla f, \nabla f)$ from \eqref{e:nablaT0};  we cannot obtain in this way explicit formulas for the $\nabla T^0(Y,I_j \nabla f, I_i \nabla f)$. However, with the help of the previous Lemmas, we will find a certain relation between the torsion $T^0$, the normalized qc-scalar curvature $S$ and their derivatives, which will lead to a system that  can be solved for the ''unknown'' components. To formulate it, we need the following covariant tensors that will also play a prominent role in the rest of the paper,
\begin{multline}\label{e:gamma}
\Gamma^i(Y,X) \overset{def}{=} {\omega_j}(Y,X)\rho_k(I_i \nabla f, \xi_i) - \omega_k(Y,X)\rho_j(I_i \nabla f, \xi_i)\\
  + df(I_kY)\rho_j(I_iX, \xi_i)  +df(I_k X)\rho_j(I_i Y, \xi_i)
- df(I_j Y)\rho_k(I_iX, \xi_i)  - df(I_j X) \rho_k(I_iY,\xi_i),
\end{multline}
with the last line equal to the symmetric part of $\Gamma^i$.
\begin{lemma}\label{l:5-3}
The following identity holds true for any cyclic permutation of the indices $(i,j,k)$,
\begin{multline}\label{e:NablaT}
5\nabla T^0(Y,X,I_i \nabla f) - 3\nabla T^0(X,Y,I_i \nabla f) = -3[\nabla T^0(\nabla f, Y, I_i X) + \nabla T^0(\nabla f, I_i Y, X)] \\
+ 3dS(Y)df(I_i X) - \frac{9}{5}dS(X) df(I_iY) + \frac{6}{5}(4+3S){f_i}g(Y,X)  + 12 \sum_{s=1}^3 \nabla^2 f(\xi_i, \xi_s){\omega_s}(Y,X) \\
-\frac{12}{5}(1+2S)\left[f{\omega_i}(X,Y) + \sum_{s=1}^3 {f_s}{\omega_s}(Y,I_i X) \right] + f[ 5T^0(X, I_i Y) - 3 T^0(I_i X, Y)]\\
+{f_i} [6T^0(I_i X, I_i Y) - 8 T^0(X,Y)] + {f_j}[5T^0(X, I_k Y) + 6T^0(I_j X, I_i Y) + 3T^0(I_k X, Y)]\\
+ {f_k}[6T^0(I_k X, I_i Y) - 5T^0(X, I_j Y) - 3T^0(I_j X, Y)]  -12\Gamma^i(Y,X).
\end{multline}
\end{lemma}

\begin{proof}
We begin by  finding another  formula for $\nabla^3 f(Y,X, \xi_i)$, see \eqref{e:nabla3f2} below, besides the already known identity \eqref{e:nabla3f1}. We begin by using the third order Ricci identity
\begin{multline}\label{e:Ricid6}
\nabla^3 f(Y,X,\xi_i) = \nabla^3 f(\xi_i, Y,X) + \nabla^2 f(T(\xi_i,Y),X) + \nabla^2 f(Y, T(\xi_i,X)) \\
+ df((\nabla_Y T)(\xi_i,X)) + R(\xi_i,Y,X,\nabla f).
\end{multline}
Next, we compute each of the terms in the right-hand side of \eqref{e:Ricid6} separately. The first term can be  simplified with the help of \eqref{e:Nab3f}, which gives
\[
\nabla^3 f(\xi_i, Y,X) = -{f_i}g(Y,X) - \sum_{s=1}^3 \nabla^2 f(\xi_i, \xi_s)\,{\omega_s}(Y,X).
\]
The Hessian equation \eqref{e:Hessf} and \eqref{e:Ricci and TorT0}  show that
\[
\nabla^2 f(T(\xi_i,Y),X) = \frac{1}{4} f[T^0(I_i Y, X) + T^0(Y, I_iX)] - \frac{1}{4}\sum_{t=1}^3 df(\xi_t)[T^0(I_i Y, I_t X) + T^0(Y, I_i I_t X)].
\]
The third term in  the right-hand side of \eqref{e:Ricid6} is handled similarly. Next, use \eqref{e:Tss} to simplify the sum of the above formulas for the  second and third terms, which give
\begin{multline}\label{e:nab2fsum}
\nabla^2 f(T(\xi_i,Y),X) + \nabla^2 f(Y, T(\xi_i,X))  = \frac{1}{2}f[T^0(I_i X, Y) + T^0(X, I_i Y)] \\
+ \frac{1}{2}{f_j}[T^0(X, I_k Y) - T^0(I_k X, Y)] + \frac{1}{2}{f_k}[T^0(I_j X, Y) - T^0(X, I_j Y)].
\end{multline}
To simplify the fourth term, we differentiate \eqref{TorT0}, using \eqref{e:Biq}, which gives
\begin{equation}\label{e:dfnabT}
df((\nabla_Y T)(\xi_i,X))=(\nabla_Y T)(\xi_i,X,\nabla f) = -\frac{1}{4}[\nabla T^0(Y, I_i X,\nabla f) + \nabla T^0(Y,X,I_i\nabla f)].
\end{equation}
At this point we need the following general formula for the curvature, cf. \cite{IV, IV2},
\begin{multline}\label{e:genR}
R(\xi_i, Y,X,Z) =  -\frac{1}{4}[\nabla T^0(X, I_i Z, Y)
+ \nabla T^0(X, Z, I_i Y)]\\
+\frac{1}{4}[\nabla T^0(Z, I_i X, Y) + \nabla T^0(Z, X, I_i Y)] +{\omega_j}(Y,X)\rho_k(I_i Z, \xi_i) - \omega_k(Y,X)\rho_j(I_i Z, \xi_i)\\
+ \omega_k(Y,Z)\rho_j(I_i X, \xi_i)+ \omega_k(X,Z)\rho_j(I_i Y, \xi_i)
-{\omega_j}(Y,Z)\rho_k(I_i X, \xi_i)- {\omega_j}(X,Z)\rho_k(I_i Y, \xi_i) .
\end{multline}
Letting $Z=\nabla f$ in \eqref{e:genR} gives the following formula for the fifth term in  the right-hand side of \eqref{e:Ricid6},
\begin{multline*}
R(\xi_i, Y, X, \nabla f) = -\frac{1}{4}[\nabla T^0(X, I_i \nabla f, Y) + \nabla T^0(X, \nabla f, I_i Y)] \\
+\frac{1}{4}[\nabla T^0(\nabla f, I_i X, Y) + \nabla T^0(\nabla f, X, I_i Y)] +\Gamma^i(Y,X),
\end{multline*}
recalling the tensor $\Gamma^i(Y,X)$ defined in \eqref{e:gamma}. A substitution of the above identities  into \eqref{e:Ricid6} yields the sought formula for $\nabla^3 f(Y,X, \xi_i)$, i.e.,

\begin{multline}\label{e:nabla3f2}
\nabla^3 f(Y,X,\xi_i) = -{f_i}g(Y,X) - \sum_{s=1}^3\nabla^2 f(\xi_i, \xi_s)\omega_s(Y,X) + \frac{1}{2}f[T^0(I_i X, Y) + T^0(X, I_i Y)] \\
+ \frac{1}{2}{f_j}[T^0(X, I_k Y) - T^0(I_k X, Y)] + \frac{1}{2}{f_k}[T^0(I_j X, Y) - T^0(X, I_j Y)] \\
 - \frac{1}{4}[\nabla T^0(Y, I_i X, \nabla f) + \nabla T^0(Y,X, I_i \nabla f)] - \frac{1}{4}[\nabla T^0(X, I_i \nabla f, Y) + \nabla T^0(X, \nabla f, I_i Y)] \\
 + \frac{1}{4}[\nabla T^0 (\nabla f, I_i X, Y) + \nabla T^0(\nabla f, X, I_i Y)] + \Gamma^i(Y,X).
\end{multline}

After this initial calculation, we use \eqref{e:nablaT0} and the symmetry in the last two indices of $\nabla T^0$ to rewrite and expand the terms $\nabla T^0(Y, I_i X, \nabla f)$ and $\nabla T^0(X, \nabla f, I_i Y)$ in \eqref{e:nabla3f2}. After a small simplification we obtain
\begin{multline}\label{g:1}
\nabla^3 f(Y, X, \xi_i) = -\frac{1}{10}(4+3S){f_i}g(Y,X) - \sum_{s=1}^3\nabla^2 f(\xi_i ,\xi_s)\omega_s(Y,X) \\
+\frac{1}{4}f[T^0(I_i X, Y) + T^0(X, I_i Y)]  +\frac{3}{20}[dS(Y)df(I_i X) + dS(X)df(I_i Y)]  - \frac{1}{2}{f_i}T^0(I_i X, I_i Y)\\
-\frac{1}{4}{f_j}[T^0(I_j Y, I_i X) + T^0(I_j X, I_i Y) -2T^0(X, I_k Y) + 2 T^0(I_k X, Y)] \\
- \frac{1}{4}{f_k}[T^0(I_k Y, I_i X) + T^0(I_k X, I_i Y) -2T^0(I_j X, Y) +2T^0(X, I_j Y)] \\ -\frac{1}{4}[\nabla T^0(Y,X,I_i \nabla f) + \nabla T^0(X,Y, I_i \nabla f)] +\frac{1}{4}[\nabla T^0(\nabla f, I_i X, Y) + \nabla T^0(\nabla f, X, I_i Y)] + \Gamma^i(Y,X).
\end{multline}

Next, we subtract \eqref{e:nabla3f1} from \eqref{g:1} and collect the terms containing  $\nabla T^0( \cdot\,,\cdot\,,I_i \nabla f)$ on one side leaving the terms containing the ''unknown'' components of $\nabla T^0$ and the vertical Hessian of $f$ on the other side. With the help of \eqref{e:Tss} we simplify the bracketed terms multiplying the vertical derivatives of $f$, which gives the claimed formula \eqref{e:NablaT}.
\end{proof}

For  several calculations we will need the symmetric part of \eqref{e:NablaT}, which is given by
\begin{multline}\label{e:NablaTsym}
\nabla T^0(Y, X, I_i \nabla f) + \nabla T^0(X, Y, I_i \nabla f) = -3[\nabla T^0(\nabla f, I_i Y, X) +\nabla T^0(\nabla f, Y, I_i X)] \\
 + \frac{3}{5}[dS(Y)df(I_i X) + dS(X)df(I_i Y)] -\frac{6}{5}(S-2)g(Y,X) f_i\\
 + f[T^0(Y, I_i X) + T^0(I_i Y, X)] +f_i [6T^0(I_i Y, I_i X) - 8T^0(Y, X)] \\
 +f_j[4T^0(I_k Y, X) + 4 T^0(Y, I_k X) + 3T^0(I_i Y, I_j X) + 3 T^0(I_j Y, I_i X)] \\
 +f_k[3T^0(I_k Y, I_i X) + 3 T^0(I_i Y, I_k X) - 4 T^0(Y, I_j X) - 4T^0(I_j Y, X)] \\
  - 12\left[  df(I_kY)\rho_j(I_iX, \xi_i)  +df(I_k X)\rho_j(I_i Y, \xi_i)
- df(I_j Y)\rho_k(I_iX, \xi_i)  - df(I_j X) \rho_k(I_iY,\xi_i)\right],
\end{multline}
taking into account
\begin{multline*}
\sum_{s=1}^3 f_s[{\omega_s}(Y, I_i X) + {\omega_s}(X, I_i Y)] = f_i[{g}(I_iY, I_i X) + {g}(I_iX, I_i Y)] + f_j[{g}(I_jY, I_i X) + {g}(I_jX, I_i Y)]\\
+ f_k[{g}(I_kY, I_i X) + {g}(I_kX, I_i Y)]
=f_i[{g}(Y,  X) + {g}(X,  Y)] + f_j[{g}(Y, I_k X) - {g}(I_kX, I_i Y)]\\
+f_k[{g}(I_jY,  X) - {g}(X, I_j Y)]=2f_i{g}(X,  Y).
\end{multline*}

\subsection{Unique continuation and a special frame}
Let $h$ be the Riemannian metric \eqref{e:h metric} and $\Delta^h$ be the associated elliptic Laplacian. In the following lemma we will give a version of \cite[Lemma 3.6]{IPV3} for the case $n=1$. In particular, this will allow the construction at almost every point of $M$ of a global orthonormal frame of the horizontal space using the horizontal gradient of $f$.

\begin{lemma}\label{elliplem}
The eigenfunction $f$ obeys the following identity
\begin{equation}\label{e:ellip}
\Delta^h f = \left( \frac{19+8S}{5} \right)f - \frac{2}{5}dS(\nabla f).
\end{equation}
In particular, $f$ and its horizontal gradient  $\nabla f$  do not vanish on any open set. Thus, if we let
\[
I_0\overset{def}{=} \id{_H}, \qquad \sigma_\alpha\overset{def}{=}  |\nabla f|^{-1} I_\alpha \nabla f,
\]
then $\{\sigma_\alpha\}_{\alpha = 0}^3$ is an orthonormal frame for the horizontal space $H$ at almost every point of $M$.
\end{lemma}

\begin{proof}
Following \cite[Lemma $5.1$]{IPV1} the Riemannian Laplacian $\Delta^h$ and the sub-Laplacian $\Delta$ are related by
\begin{equation}\label{e:RiemLap}
\Delta^h f = \Delta f - \sum_{s=1}^3 \nabla^2 f(\xi_s, \xi_s).
\end{equation}
Taking the trace $X={e_\gamma}$, $Y=I_i {e_\gamma}$ of \eqref{e:NablaT} using that $T^0$ is completely trace-free gives
\begin{multline}\label{e:NablaT01i}
5\nabla T^0(I_i {e_\gamma}, {e_\gamma}, I_i \nabla f) - 3\nabla T^0({e_\gamma}, I_i {e_\gamma}, I_i \nabla f) = 3dS(I_i {e_\gamma})df(I_i {e_\gamma})
+ \frac{9}{5}dS({e_\gamma})df({e_\gamma}) \\ -\frac{48}{5}(1+2S)f
 +48 \nabla^2 f(\xi_i, \xi_i) - 12 \Gamma^i(I_i {e_\gamma}, {e_\gamma}).
\end{multline}
From $\nabla T^0({e_\gamma}, I_i {e_\gamma}, X) = - \nabla T^0(I_i {e_\gamma}, {e_\gamma}, X)$ and $\Gamma^i(I_i {e_\gamma}, {e_\gamma})=0$ by \eqref{e:gamma} we can solve for the component of the vertical Hessian of $f$ in \eqref{e:NablaT01i} which gives
\begin{equation}\label{e:vertHessf}
\nabla^2 f(\xi_i, \xi_i) = \frac{1}{6}\nabla T^0({e_\gamma}, I_i {e_\gamma}, I_i \nabla f) + \frac{1}{10}dS(\nabla f) - \frac{1}{5}(1+2S)f.
\end{equation}
Using \eqref{e:Tss} in which we take $X = I_i {e_\gamma}$, $Y=I_i \nabla f$, and \eqref{e:Biq}, gives the following trace formula
\begin{equation}\label{e:8}
\nabla T^0({e_\gamma}, {e_\gamma}, \nabla f) + \sum_{s=1}^3 \nabla T^0({e_\gamma},I_s {e_\gamma}, I_s \nabla f) =0.
\end{equation}
\noindent Then, \eqref{e:nablaT0} with $X=Y={e_\gamma}$ shows that the divergence of $T^0$ satisfies
\begin{equation}\label{e:divT0}
\nabla T^0({e_\gamma}, {e_\gamma}, \nabla f) = -\frac{3}{5}\left[dS(\nabla f) - 4(S-2)f\right].
\end{equation}
\noindent Therefore, \eqref{e:vertHessf} with \eqref{e:8} and \eqref{e:divT0} implies
\begin{equation}\label{e:7}
\sum_{s=1}^3 \nabla^2 f(\xi_s, \xi_s) = \frac{2}{5}dS(\nabla f) + \frac{1}{5}(1-8S)f.
\end{equation}
\noindent A substitution of \eqref{e:7} into \eqref{e:RiemLap}, taking into account that  $\Delta f = 4f$, shows \eqref{e:ellip}.

The final part of the Lemma follows from Aronszajn's unique continuation result \cite{Aron}.
\end{proof}

\subsubsection{Components of the torsion } Since in lemma \ref{elliplem} we  found a global orthonormal frame  $\{\sigma_\alpha,\, \xi_s\}$,  it will be convenient to use the index notation for the components of the involved tensors constructed from the torsion as follows
\begin{equation}\label{sigalph}
T_{\alpha\beta}\overset{def}{=} T^0(I_\alpha \nabla f, I_\beta \nabla f), \qquad \nabla T^0(I_{\gamma} \nabla f, I_{\alpha} \nabla f, I_{\beta} \nabla f) = T_{\alpha\beta;\gamma},
\end{equation}
where $I_0\overset{def}{=} \id{_H}$.
In particular, the fact that $T^0$ is a symmetric tensor can be written as $T_{\alpha\beta} = T_{\beta\alpha}$ and \eqref{e:Tss} becomes
\begin{equation}\label{e:Taa}
T_{00} + T_{11} + T_{22} + T_{33} = 0.
\end{equation}
Furthermore, from the properties of the connection  we have $T_{\alpha\beta;\gamma} = T_{\beta\alpha;\gamma}$.

Next, we will  show that  $T_{i0;0}$ vanish. This will yield a relation between the vertical derivatives $f_s$ and the torsion components $T_{\alpha\beta}$.

\begin{lemma}
The following identities between the components $T_{\alpha\beta}$ of the torsion tensor and the vertical derivatives $f_i$ of the eigenfunction $f$ hold true
\begin{equation}\label{e:T0i0}
f_s T_{00} = f_1 T_{s1} + f_2 T_{s2} + f_3 T_{s3},\qquad  s=1,2,3,
\end{equation}
and
\begin{equation}\label{e:T0i00}
T_{i0;0} = 0.
\end{equation}
\end{lemma}

\begin{proof}
\noindent From \eqref{e:nabla2f1} and \eqref{e:T0Xf} we have
\begin{equation}\label{e:nabla2f2}
\nabla^2 f(X,\xi_i)= df(I_i X) - \frac{2}{3}T^0(I_i X, \nabla f)  - \frac{2}{3}T^0(X,I_i \nabla f).
\end{equation}
By \eqref{e:T0i} we have $T_{i0}=0$, hence   the above identity shows $\nabla^2 f(\nabla f, \xi_i) = 0$. Therefore, we have
\begin{equation}\label{e:4}
\nabla^2 f(\nabla f, \xi) = \nabla^2 f(\nabla f, \nabla_A \xi)=0, \quad \xi\in V, A\in TM,
\end{equation}
taking into account \eqref{e:Biq}. The covariant derivative along $\nabla f$ of the identity $\nabla^2 f(\nabla f, \xi_i) = 0$  gives
\begin{multline*}
0=\nabla^3 f(\nabla f, \nabla f, \xi_i) -f\,\nabla^2 f(\nabla f, \xi_i)-\sum_{s=1}^3f_s\nabla^2 f(I_s\nabla f, \xi_i)-\nabla^2 f(\nabla f, \nabla_{\nabla f}\xi_i)\\
=\nabla^3 f(\nabla f, \nabla f, \xi_i) -\sum_{s=1}^3f_s\left[df(I_i I_s\nabla f) - \frac{2}{3}T^0(I_i I_s\nabla f, \nabla f)  - \frac{2}{3}T^0(I_s\nabla f,I_i \nabla f)\right]\\
= \nabla^3 f(\nabla f, \nabla f, \xi_i) + |\nabla f|^2 f_i + \frac{2}{3}[- f_i T_{00} +f_i T_{ii} + f_j T_{ji} +f_k T_{ki} ]
\end{multline*}
 using  the Hessian equation \eqref{e:Hessf}, \eqref{e:4}, $T_{i0}=0$ by \eqref{e:T0i}, and \eqref{e:nabla2f2}.
However, it follows from \eqref{e:nablaT0} that
\begin{equation}\label{e:Ti0,0}
T_{i0;0} = -f_i T_{00} + f_i T_{ii} + f_j T_{ji} + f_k T_{ki},
\end{equation}
hence,
\begin{equation}\label{e:5}
\nabla^3 f(\nabla f, \nabla f, \xi_i) + |\nabla f|^2 f_i + \frac{2}{3}T_{i0;0} = 0.
\end{equation}

On the other hand, from \eqref{g:1} we have
\begin{multline}\label{e:6}
\nabla^3 f(\nabla f,\nabla f,\xi_i) =  -\frac{1}{5}(1+2S)|\nabla f|^2f_i \\
 +\frac{2}{3}\left[ fT^0(\nabla f, I_i \nabla f)- {f_i}T^0(\nabla f,  \nabla f)+ {f_j}T^0(\nabla f, I_k \nabla f)-{f_k}T^0(\nabla f, I_j \nabla f) \right]\\
+\frac{2}{5}dS(\nabla f)df(I_i \nabla f)-\frac{2}{3}\nabla T^0(\nabla f, \nabla f, I_i \nabla f)
=-\frac{1}{5}(1+2S)|\nabla f|^2f_i - \frac{2}{3}f_iT_{00}-\frac{2}{3}T_{i0;0}
\end{multline}
using $T_{i0}=0$ by \eqref{e:T0i} to obtain the last equality.
Now, \eqref{e:5} and \eqref{e:6} give
\begin{equation}\label{e:3}
\nabla^3 f (\nabla f, \nabla f, \xi_i) = - \frac{1}{2}T_{i0;0} - |\nabla f|^2 f_i .
\end{equation}
A substitution of \eqref{e:3} into \eqref{e:5} shows \eqref{e:T0i00}, which together with \eqref{e:Ti0,0} give \eqref{e:T0i0}.
\end{proof}

\subsubsection{{\bf{The components}} \boldmath$T_{ij;0}$ {\bf{and the qc-Ricci $2$-forms}}} At this stage, from \eqref{e:nablaT0} we can compute the components $T_{\alpha\beta;\gamma}$ only when either $\alpha = 0$ or $\beta = 0$. However, by evaluating \eqref{e:NablaT} on the $\{\sigma_\alpha\}_{\alpha =0}^3$ frame, cf. Lemma \ref{elliplem}, we will be able to use the components $T_{0i;j}$ to determine not only  $T_{ij; 0}$, but also the qc-Ricci $2$-forms $\rho_s$ defined in \eqref{qc-Ric}. We will use  the following identities for the qc-Ricci $2$-forms, cf. \cite[Theorem $3.1$]{IV} or \cite[Theorem $4.3.11$]{IV2},
\begin{equation}\label{e:2.12}
\begin{aligned}
18\rho_s(\xi_s, X) & = 3dS(X) + \frac{1}{2}\nabla T^0({e_\gamma}, {e_\gamma}, X) - \frac{3}{2} \nabla T^0({e_\gamma}, I_s {e_\gamma}, I_s X), \\
18\rho_i(\xi_j, I_k X) & =-18\rho_i(\xi_k, I_j X)= 3dS(X) - \frac{5}{2}\nabla T^0({e_\gamma}, {e_\gamma}, X) -\frac{3}{2}\nabla T^0({e_\gamma}, I_i {e_\gamma}, I_i X).
\end{aligned}
\end{equation}
We begin by using \eqref{e:NablaTsym} and the symmetry of the $T_{\alpha\beta;\gamma}$ in the first two indices to prove the following Lemma.
\begin{lemma}
We have
\begin{equation}\label{e:rhokjj2}
 \rho_k(I_j \nabla f, \xi_j) = -\frac{3}{5}(S-2)f_k,
\end{equation}
\begin{equation}\label{e:Tij0}
T_{ij;0} = \frac{1}{4}[ f_k(T_{ii} - T_{jj}) + f_j T_{kj} - f_i T_{ki}].
\end{equation}
\end{lemma}

\begin{proof}
Letting $Y = I_j\nabla f$ and $X = \nabla f$ in \eqref{e:NablaTsym} we have, taking into account $T_{i0}=0$ and $T_{i0;0} = 0$ by \eqref{e:T0i} and  \eqref{e:T0i00}, respectively, the following identity
\begin{multline*}
T_{i0;j}+T_{j0;i}=-3T_{ji;0}+fT{ji}+6f_iT_{ki}+7f_jT_{jk}+f_k(-3T_{ii}+3T_{kk}-4T_{jj}+4T_{00})\\
-12|\nabla f|^2 \rho_k(I_i \nabla f, \xi_i).
\end{multline*}
From \eqref{e:nablaT0} we can find another formula for $T_{i0;j}+T_{j0;i}$, which together with the above identity and  \eqref{e:T0i0} gives
\begin{equation}\label{e:67.5}
T_{ji;0} = -3|\nabla f|^2 \rho_k(I_i \nabla f, \xi_i) + \frac{3}{2}f_k T_{00} - [f_k T_{jj} - f_j T_{kj}]  + \frac{1}{2}[f_i T_{ki} - f_k T_{ii}].
\end{equation}

On the other hand, by first taking  \eqref{e:NablaTsym} for $j$, and  then working as above but using $Y = I_i\nabla f$ and $X = \nabla f$ we obtain the identity
\begin{equation}\label{e:67}
T_{ij;0} = 3|\nabla f|^2\rho_k(I_j \nabla f, \xi_j) - \frac{3}{2}f_k T_{00}  +[ f_k T_{ii} - f_i T_{ki}] - \frac{1}{2} [f_j T_{kj} - f_k T_{jj}].
\end{equation}
Therefore, the symmetry of $T_{ij;0}$ in $i$ and $j$, together with the last line in \eqref{e:2.10}, \eqref{e:Taa} and \eqref{e:T0i0}  give
\begin{multline*}
0 = T_{ij;0} - T_{ji;0} =3|\nabla f|^2\left[\rho_k(I_j \nabla f, \xi_j) + \rho_k(I_i \nabla f, \xi_i)\right] - 3f_kT_{00} + \frac{3}{2}\left[f_k(T_{ii} + T_{jj}) - f_i T_{ki} - f_j T_{kj}\right] \\
= 6|\nabla f|^2 \rho_k(I_j \nabla f, \xi_j) - 6f_kT_{00},
\end{multline*}
which, by \eqref{e:T0Xf}, implies \eqref{e:rhokjj2}. Similarly, \eqref{e:67.5} and \eqref{e:67} yield
\[
2 T_{ij;0} = T_{ij;0} + T_{ji;0} = 3|\nabla f|^2\left[\rho_k(I_j \nabla f, \xi_j) - \rho_k(I_i \nabla f, \xi_i)\right] + \frac{1}{2}[ f_k(T_{ii} - T_{jj}) + f_j T_{kj} - f_i T_{ki}].
\]
By \eqref{e:2.10} the term in the first brackets is zero, hence  we conclude \eqref{e:Tij0}.
\end{proof}

\subsubsection{The components $T_{ii;0}$ and the vertical Hessian of $f$} With the results of the previous section, we can begin to determine the components of $dS$. In particular, we can now show that one of the components of $dS|_H$ vanishes.

\begin{lemma}
The normalized qc-scalar curvature $S$ satisfies the following relations  at almost every point of $M$
\begin{equation}\label{e:dSnabf}
dS(\nabla f) = 0, \qquad T_{00;0} = 0, \quad\text{and}\quad T_{ii;0} = \frac{1}{2}[f_jT_{ki} - f_k T_{ji}].
\end{equation}

\end{lemma}

\begin{proof}
Letting $X = I_i\nabla f$ and $Y= \nabla f$  in \eqref{e:NablaTsym} and taking \eqref{e:T0i} into account, we have the identity
\begin{equation}\label{z:1}
\begin{aligned}
T_{ii;0} + T_{0i;i} = -3[T_{ii;0} - T_{00;0}] - \frac{3}{5}|\nabla f|^2 dS(\nabla f)  - f[T_{00} - T_{ii}] + f_j T_{ki} - f_k T_{ji}.
\end{aligned}
\end{equation}
From the formula for $\nabla T^0(Y,X,\nabla f)$ in \eqref{e:nablaT0} we can compute that
\begin{equation}\label{e:Ti0i} T_{i0;i} = f[T_{ii} - T_{00}] + f_k T_{ij} - f_j T_{ki} \end{equation}

\noindent and $5T_{00;0} = -3|\nabla f|^2 dS(\nabla f)$; using this and \eqref{e:Ti0i} in \eqref{z:1} we see that
\begin{equation}\label{e:Tii0}
T_{ii;0} = -\frac{3}{5}|\nabla f|^2dS(\nabla f) + \frac{1}{2}[f_jT_{ki} - f_k T_{ji}].
\end{equation}
On the other hand, from the $Sp(1)Sp(1)$-invariance of \eqref{e:Taa} it follows that
\[ T_{00;0} + T_{11;0} + T_{22;0} + T_{33;0} = 0
\]
which together with \eqref{e:Tii0} and \eqref{e:Ti0i} yield $|\nabla f|^2 dS(\nabla f) = 0$, hence \eqref{e:dSnabf}.
\end{proof}

Now we can  determine the components of the vertical Hessian of $f$, which will then be used in the proofs that the remaining components of $dS$ vanish and, eventually, in the final section that the torsion vanishes.

\begin{lemma}
With the assumptions of Theorem \ref{t:mainpan}, if $f$ satisfies \eqref{e:Hessf} then we have the following identities for the vertical Hessian of $f$,
\begin{equation}\label{e:vHessfii}
|\nabla f|^2 \nabla^2 f(\xi_i, \xi_i) =- \frac{1}{5}(1+2S)|\nabla f|^2 f - \frac{2}{3}[fT_{ii} + f_k T_{ij} - f_j T_{ki}],
\end{equation}
\begin{equation}\label{e:vHessfij}
|\nabla f|^2 \nabla^2 f(\xi_i, \xi_j) = -\frac{2}{3} f T_{ij} +\frac{1}{2}(4-S)|\nabla f|^2 f_k - \frac{11}{12}[ f_i T_{ki} - f_k T_{ii}] -\frac{1}{4}[ f_j T_{kj} - f_k T_{jj}],
\end{equation}
\begin{equation}\label{e:vertric}
\nabla^2 f(\xi_i, \xi_j) - \nabla^2 f(\xi_j, \xi_i) = \frac{2}{5}(3+S)f_k.
\end{equation}
\end{lemma}

\begin{proof}
First, letting $X = I_i\nabla f$ and $Y = \nabla f$ in \eqref{e:NablaT} and taking into account $\Gamma^i(\nabla f, I_i \nabla f) =0$ from \eqref{e:gamma} we have

\begin{multline*}
5T_{ii;0} - 3T_{i0;i} = 3T_{00;0} - 3T_{ii;0} - 3|\nabla f|^2 dS(\nabla f) + 5fT_{ii} + 3fT_{00} \\
+ 12|\nabla f|^2\nabla^2 f(\xi_i, \xi_i)  + \frac{12}{5}(1+2S)|\nabla f|^2 f + f_k T_{ij} - f_j T_{ki}.
\end{multline*}
Using the formulas in \eqref{e:dSnabf} and \eqref{e:Ti0i}, we can expand the above
\begin{multline}
\frac{5}{2}[f_jT_{ki} - f_k T_{ji}] - 3(f[T_{ii} - T_{00}] + f_k T_{ij} - f_j T_{ki}) =  -\frac{3}{2}[f_j T_{ki} - f_k T_{ji}]+ 5fT_{ii} + 3fT_{00} \\+ 12|\nabla f|^2 \nabla^2 f(\xi_i , \xi_i) + \frac{12}{5}(1+2S)|\nabla f|^2 f + f_k T_{ij} - f_j T_{ki}
\end{multline}
and then solve this for the vertical Hessian of $f$ to obtain \eqref{e:vHessfii}. Next, let $X = I_j\nabla f$ and $Y = \nabla f$ in \eqref{e:NablaT}; then use that $\Gamma^i(\nabla f, I_j \nabla f) = 2|\nabla f|^2\rho_k(I_i \nabla f, \xi_i)$ from \eqref{e:gamma} to see
\begin{multline*}
5T_{ji;0} - 3T_{0i;j} = -3T_{0k;0} - 3 T_{ij;0} + 5fT_{ji} + 12|\nabla f|^2\nabla^2 f(\xi_i ,\xi_j) -\frac{12}{5}(1+2S)|\nabla f|^2 f_k \\
-24|\nabla f|^2\rho_k(I_i\nabla f, \xi_i)  + 6f_i T_{ki} + 5f_j T_{jk} -6f_k T_{ii} - 5 f_k T_{jj} + 3 f_k T_{00}.
\end{multline*}
From \eqref{e:nablaT0} with $Y = I_j \nabla f$, $X= I_i \nabla f$ we can compute that $T_{0i;j} = fT_{ij} + f_i T_{ki} + f_k[T_{00} - T_{ii}]$. Then this, along with $T_{i0;0}=0$ from \eqref{e:T0i00}, the formula for $\rho_k(I_j\nabla f, \xi_j)$ in \eqref{e:rhokjj2}, and for $T_{ij;0}$ in \eqref{e:Tij0}, applied to the above gives
\begin{multline}
\frac{5}{4}[f_k(T_{ii} - T_{jj}) + f_j T_{kj} - f_i T_{ki}] - 3[fT_{ij} + f_i T_{ki} + f_k (T_{00} - T_{ii})] \\= - \frac{3}{4}[f_k(T_{ii} - T_{jj}) + f_j T_{kj} - f_i T_{ki}] + 5fT_{ji} + 12|\nabla f|^2 \nabla^2 f(\xi_i , \xi_j) - \frac{12}{5}(1+2S)|\nabla f|^2 f_k \\+ \frac{72}{5}(S-2)|\nabla f|^2f_k +6 f_i T_{ki} + 5f_j T_{jk} - 6 f_k T_{ii} - 5 f_k T_{jj} + 3 f_k T_{00}.
\end{multline}
Using that $5T_{00} = - 3(S-2)|\nabla f|^2$ from \eqref{e:T0Xf} and solving the above for $|\nabla f|^2 \nabla^2 f(\xi_i, \xi_j)$ yields \eqref{e:vHessfij}.

Finally, recall that $\{\xi_s\}_{s=1}^3$ is an orthornormal frame for $V$  with respect to the Riemannian metric \eqref{e:h metric}. With this, and the orthonormal frame $\{\sigma_\alpha\}_{\alpha = 0}^3$ for $H$, we can expand
\[
T(\xi_i, \xi_j) =|\nabla f|^{-2}\sum_{\alpha=0}^3 h(T(\xi_i, \xi_j), I_\alpha \nabla f)I_\alpha \nabla f + \sum_{s=1}^3 h(T(\xi_i, \xi_j),\xi_s)\xi_s.
\]
By the last two lines of \eqref{e:2.10} we have $h(T(\xi_i, \xi_j), I_\alpha \nabla f) = -\rho_k(I_i I_\alpha \nabla f, \xi_i)$ and $h(T(\xi_i,\xi_j),\xi_s) = -S \delta_{ks}$. Thus, \eqref{e:rhokjj2} and the Ricci identity $\nabla^2 f(\xi_i, \xi_j) - \nabla^2 f(\xi_j, \xi_i) = -df(T(\xi_i, \xi_j))$, shows \eqref{e:vertric}.
\end{proof}

\subsection{The qc-scalar curvature is constant}\label{ss:S is const} Here we obtain first a formula for the horizontal part $dS|_H$ of the differential of $S$ and therefore one for the horizontal Hessian $\nabla^2 S(X,Y)$ as well. The latter will then be used to show that $dS|_V = dS|_H = 0$ and allow us to conclude that $S$ is constant.

Several divergences of the torsion tensor $T^0$ will appear in the next calculations, so we remind the notation set in \eqref{e:def of div}. In particular, we will use that if  $\alpha\neq 0 $ then $\nabla_\alpha^*T^0(X) = -\nabla T^0(I_\alpha {e_\gamma}, {e_\gamma}, X).$

\begin{lemma}
The next identity holds at almost every point of $M$,
\begin{equation}\label{e:dSi}
dS(I_t \nabla f) = -2(S-2)f_t, \qquad t = 1,2,3.
\end{equation}
\end{lemma}

\begin{proof}
With \eqref{e:rhokjj2}, the last line of \eqref{e:2.10}, and second line of \eqref{e:2.12}, we arrive at the identity
\begin{equation}\label{e:rhoikk1}
-\frac{3}{5}(S-2) f_i = \rho_i(I_k \nabla f, \xi_k) = -\frac{1}{6}dS(I_i\nabla f) + \frac{5}{36}\nabla^* T^0(I_i \nabla f) - \frac{1}{12}\nabla_i^* T^0(\nabla f).
\end{equation}
By \eqref{e:nablaT0} and the fact that $T^0$ is completely trace-free we have the identity
\begin{equation}\label{e:diviT0}
\nabla_i^* T^0(\nabla f) = \frac{3}{5}[dS(I_i \nabla f) + 4(S-2)f_i].
\end{equation}
Therefore, we need only determine $\nabla^* T^0(I_i \nabla f)$. For this, take the trace $X = {e_\gamma}$, $Y= I_\alpha {e_\gamma}$ in \eqref{e:NablaT}:
\begin{multline}\label{e:NablaT01}
5\nabla T^0(I_\alpha {e_\gamma}, {e_\gamma}, I_i \nabla f) - 3\nabla T^0({e_\gamma}, I_\alpha {e_\gamma}, I_i \nabla f) = 3dS(I_\alpha {e_\gamma})df(I_i {e_\gamma}) - \frac{9}{5}dS({e_\gamma})df(I_i I_\alpha {e_\gamma}) \\
+ \frac{6}{5}(4+3S)g(I_\alpha {e_\gamma}, {e_\gamma})f_i -\frac{12}{5}(1+2S)\left( f{\omega_i}({e_\gamma}, I_\alpha {e_\gamma}) + \sum_{s=1}^3 f_s\,{\omega_s}(I_\alpha {e_\gamma}, I_i {e_\gamma})\right) \\
+12 \sum_{s=1}^3 \nabla^2 f(\xi_i, \xi_s) {\omega_s}(I_\alpha {e_\gamma}, {e_\gamma}) - 12 \Gamma^i(I_\alpha {e_\gamma}, {e_\gamma}).
\end{multline}
For $\alpha = 0$ equation \eqref{e:NablaT01} becomes
\begin{equation}\label{e:divT0i1}
\nabla^* T^0(I_i \nabla f) = -\frac{3}{5}[dS(I_i \nabla f) + 4(S-2)f_i] - 12\Gamma^i({e_\gamma}, {e_\gamma}).
\end{equation}
Using \eqref{e:gamma} and \eqref{e:2.12} we see that
\begin{multline}\label{e:rhojrhok}
\Gamma^i({e_\gamma}, {e_\gamma}) = 2[\rho_j(I_j \nabla f, \xi_i) + \rho_k(I_k\nabla f, \xi_i)] = \frac{2}{3}dS(I_i \nabla f) - \frac{5}{9}\nabla^* T^0(I_i \nabla f)\\
 + \frac{1}{6}[\nabla_j^* T^0(I_k \nabla f) - \nabla_k^* T^0(I_j \nabla f)],
\end{multline}
thus a  substitution of \eqref{e:rhojrhok} into \eqref{e:divT0i1} gives
\begin{equation}\label{e:divT0i}
\nabla^* T^0(I_i \nabla f) = \frac{3}{7}\left(\frac{23}{5} dS(I_i \nabla f) + \frac{12}{5}(S-2)f_i + [\nabla_j^* T^0(I_k \nabla f)-\nabla_k^* T^0(I_j \nabla f)]\right).
\end{equation}

Now we write \eqref{e:NablaT01} for $j$ instead of  $i$ and then let $\alpha = k$ in the result. Then we use \eqref{e:gamma} to see that $\Gamma^j(I_k {e_\gamma}, {e_\gamma}) = -4\rho_i(I_j \nabla f, \xi_j)$ and thus by \eqref{e:rhokjj2}:
\begin{equation}\label{e:divT0kj}
\nabla_k^* T^0(I_j \nabla f) = -\frac{3}{5}dS(I_i \nabla f) - \frac{6}{5}(7-S)f_i + 6\nabla^2 f(\xi_j, \xi_k).  \end{equation}

Next, we do one more permutation of the indices and consider \eqref{e:NablaT01} for $k$ instead of  $i$  and then let $\alpha = j$, which taking into account $\Gamma^k(I_j {e_\gamma}, {e_\gamma}) = 4\rho_i(I_k\nabla f, \xi_k)$ and \eqref{e:rhokjj2} gives an identity for the remaining divergence
\begin{equation}\label{e:divT0jk}
\nabla_j^* T^0(I_k \nabla f) = \frac{3}{5}dS(I_i \nabla f) + \frac{6}{5}(7-S)f_i + 6 \nabla^2 f(\xi_k, \xi_j). \end{equation}

Therefore, subtracting \eqref{e:divT0kj} from \eqref{e:divT0jk} and applying \eqref{e:vertric} we come to
\begin{equation}\label{e:diffdiv1}
\nabla_j^* T^0(I_k \nabla f)-\nabla_k^* T^0(I_j \nabla f) = \frac{6}{5}[dS(I_i \nabla f) -4(S-2)f_i].
\end{equation}

Lastly, a substitution of \eqref{e:diffdiv1} into \eqref{e:divT0i} gives
\begin{equation}\label{e:divT0i2}
\nabla^*T^0(I_i \nabla f) = \frac{3}{7}\left( \frac{29}{5}dS(I_i \nabla f) - \frac{12}{5}(S-2)f_i\right)
\end{equation}
which after using it together with  \eqref{e:diviT0}  in \eqref{e:rhoikk1} shows \eqref{e:dSi}.
\end{proof}

\begin{lemma}\label{Sconst}
The normalized qc-scalar curvature $S$ is constant, in fact $S=2$. In particular,
\begin{equation}\label{Twhens=2}
T_{00} =0, \qquad T_{11} + T_{22} + T_{33} = 0,  \quad\text{and}\quad f_1T_{s1} + f_2 T_{s2} + f_3 T_{s3} = 0,\quad s= 1,2,3,
\end{equation}
and for any cyclic permutation $(i,j,k)$ of $(1,2,3)$ we have
\begin{equation}\label{e:f's and T's}
 f_k T_{ik} - f_i T_{kk} = f_i T_{jj} - f_j T_{ij}.
\end{equation}
\end{lemma}

\begin{proof}
First we will show that the differential of $S$ vanishes on the vertical space, $dS|_V = 0$. With \eqref{e:dSnabf} and \eqref{e:dSi} we can write the horizontal gradient of $S$ in the $\{\sigma_\alpha\}_{\alpha =0}^3$ frame in the form
\begin{equation}\label{e:nabS}
|\nabla f|^{2} \nabla S  =-2(S-2)\sum_{t=1}^3 f_t \,I_t\nabla f.
\end{equation}
The covariant derivative of \eqref{e:nabS} along a horizontal vector $Y$, using \eqref{e:Biq}, the horizontal Hessian equation \eqref{e:Hessf}, and \eqref{e:nabla2f1} for the term $\nabla^2 f(Y,\xi_i)$, gives the equation
\begin{multline}\label{e:HessS}
\frac{1}{2}|\nabla f|^2 \nabla^2 S(Y,X) = f df(Y)dS(X) + \sum_{t=1}^3 f_t [df(I_t Y) dS(X) + df(I_t X)dS(Y)] \\
+ (S-2) \sum_{t=1}^3\left[f_t\nabla^2 f(Y, I_t X) + \left(\frac{1}{5}(1+2S)df(I_t Y) - \frac{2}{3}T^0(Y, I_t\nabla f) \right)df(I_t X)\right].
\end{multline}
Using \eqref{e:Hessf} again, the identities in \eqref{e:dSnabf} and \eqref{e:dSi} we find $\nabla^2 S(I_i \nabla f, \nabla f) = \nabla^2 S(\nabla f, I_i \nabla f)$. Hence, by the Ricci identity $\nabla^2 S(X,Y) - \nabla^2 S(Y,X) = -2 \sum_{t=1}^3 {dS(\xi_t)}{\omega_t}(X,Y) $ we have
\[
-2\sum_{t=1}^3 dS(\xi_t)\omega_t(I_i \nabla f, \nabla f) =\nabla^2 S(I_i \nabla f, \nabla f) - \nabla S(\nabla f, I_i \nabla f) =0,
\]
which implies
\begin{equation}\label{e:dSxi}
dS(\xi_t) =0,\qquad t= 1,2,3.
\end{equation}

Now we can show that the differential of $S$ vanishes on all horizontal vectors as well,  $dS|_H=0$. From \eqref{vderiv} and \eqref{e:dSxi} we find $\nabla^2 S({e_\gamma}, I_i {e_\gamma}) = 0$. On the other hand, using \eqref{e:HessS} with \eqref{e:dSi} we also have
\[
|\nabla f|^2 \nabla^2 S({e_\gamma}, I_i {e_\gamma}) = - 2f dS(I_i \nabla f).
\]
Thus, since $f \neq 0$ a.e., see Lemma \ref{elliplem}, we conclude
\begin{equation}\label{e:dSt}
dS(I_t\nabla f) = 0\qquad t=1,2,3.
\end{equation}
Hence, since in addition we have $dS(\nabla f) = 0$ by \eqref{e:dSnabf},  it follows that $dS|_H=0$. Therefore, taking  into account that $dS$ vanishes on the Reeb vector fields as proven above, it follows that $dS = 0$ and hence  $S$ is constant.

In order to determine the constant we note that from \eqref{e:dSi},  either $S= 2$ or $f_1 = f_2 = f_3 = 0$ on some open set. Arguing by contradiction, suppose the latter, then for any horizontal vector $X$ we would have, by \eqref{e:Biq} and the assumption $f_s=0$, the idenity
\[
0 = X{f_i} = \nabla^2 f(X, \xi_i) - \alpha_j(X)f_k + \alpha_k(X)f_j= \nabla^2 f(X,\xi_i).
\]
Then it would follow from \eqref{e:nabla2f1} that
$10\,T_{ii} = -3(1+2S)|\nabla f|^2$. On the other hand, the component  $T_{00}$ can be computed from  \eqref{e:T0Xf}, which gives $T_{00}=-\frac 35(S-2)|\nabla f|^2$. Therefore, by \eqref{e:Taa} we have
\[
0 = \sum_{\alpha = 0}^3 T_{\alpha\alpha} = \frac{3}{10}(1-8S)|\nabla f|^2,
\]
hence, $S=1/8$. This is a contradiction since the Lichnerowicz-type bound \eqref{e:Lich} implies, due to $T^0$ being a trace-free tensor, that $S \geq 2$. Thus we must have $S=2$, and consequently \eqref{e:T0Xf} now implies $T_{00}=0$. With this, \eqref{Twhens=2} follows from \eqref{e:Taa} and \eqref{e:T0i0}.

Finally, a substitution of the second identity in \eqref{Twhens=2} into the third one  written for $s=i$ shows
\[
0 = f_iT_{ii} + f_j T_{ij} + f_k T_{ik} = f_i (-T_{jj} - T_{kk}) + f_j T_{ij} + f_k T_{ik}
\]
from which \eqref{e:f's and T's} follows.
\end{proof}

\subsection{Vanishing of the torsion}
The last application of \eqref{e:NablaT} is to finally show that $T^0 = 0$. We begin with a simple lemma  describing the consequences of $S=2$ on the components of the divergences $\nabla_i^* T^0(X){=} \nabla T^0({e_\gamma}, I_i {e_\gamma}, X)$ defined in \eqref{e:def of div}.

\begin{lemma}
The divergences of the torsion satisfy the following identities,
\begin{equation}\label{e:f1}
|\nabla f|^2 \nabla_i^* T^0(I_i \nabla f) = -4[fT_{ii} + f_k T_{ij} - f_j T_{ki}]
\end{equation}
\begin{equation}\label{e:f2}
|\nabla f|^2 \nabla_j^*T^0(I_k\nabla f) = |\nabla f|^2 \nabla_k^*T^0(I_j\nabla f) = -4[f T_{jk} + f_j T_{ij} - f_i T_{jj}].
\end{equation}
\end{lemma}
\begin{proof}
Since $S=2$ by Lemma \ref{Sconst}, equation \eqref{e:diffdiv1} implies $\nabla_kT^0(I_j \nabla f) = \nabla_j T^0(I_k\nabla f)$, which gives the first equality in \eqref{e:f2}. Furthermore, \eqref{e:vHessfij} now takes the simpler form
\begin{equation}\label{e:vHessfjkS}
12|\nabla f|^2\nabla^2 f(\xi_j, \xi_k) = -8fT_{jk} + 12|\nabla f|^2f_i - 11f_j T_{ij} +11 f_i T_{jj} - 3f_k T_{ik} +3 f_i T_{kk}.
\end{equation}
Applying \eqref{e:f's and T's} to \eqref{e:vHessfjkS} we find
\begin{equation}\label{e:vHessfjkS2}12|\nabla f|^2 \nabla^2 f(\xi_j,\xi_k) = -8fT_{jk} + 12 |\nabla f|^2 f_i - 8 f_j T_{ij} + 8 f_i T_{jj}. \end{equation}
Substituting \eqref{e:vHessfjkS2} into \eqref{e:divT0kj} yields the second equality of \eqref{e:f2}.

Finally, let $\alpha = i$ in \eqref{e:NablaT01}, which due to $ \nabla_\alpha^*T^0(X) = -\nabla T^0(I_\alpha {e_\gamma}, {e_\gamma}, X)$ takes the form (for $i$ fixed)
\begin{multline}
-8\nabla_i^*T^0(I_i\nabla f)=5\nabla T^0(I_i e_\gamma, e_\gamma, I_i \nabla f) - 3 \nabla T^0(e_\gamma, I_i e_\gamma, I_i \nabla f)\\
= -48 f - 48\nabla^2 f(\xi_i, \xi_i)
 - 12 \Gamma^i(I_i e_\gamma, e_\gamma)=-48\left[f + \nabla^2 f(\xi_i, \xi_i) \right],
\end{multline}
using $\Gamma^i(I_i {e_\gamma}, {e_\gamma}) = 0$ from \eqref{e:gamma}. An application of  \eqref{e:vHessfii} to the last  equation gives \eqref{e:f1}.
\end{proof}

In the last lemma needed for the proof of the main theorem we derive the key relation between the components of the torsion tensor. We continue the use of  the notation $T_{ij}$ for the  components of the torsion set in \eqref{sigalph}.

\begin{lemma} 
The next identities hold at almost every point,
\begin{equation}\label{e:fTjk}
fT_{jk} = \frac{1}{4}[f_i T_{kk} - f_k T_{ki}]=\frac14[f_j T_{ij} - f_i T_{jj}]
\end{equation}
\begin{equation}\label{e:fTii}
fT_{ii} = \frac{1}{4}[f_k T_{ij} - f_j T_{ki}].
\end{equation}\end{lemma}

\begin{proof}
First, let us dispose with the trivial case, by noting that the second identity in \eqref{e:fTjk} follows directly from \eqref{e:f's and T's}.

We turn to the proof of the first equality in \eqref{e:fTjk}. Let $A(Y,X)$ denote the tensor in the left-hand side of \eqref{e:NablaT}. Therefore,
$$16\nabla T^0(Y,X, I_i \nabla f)=5A(Y,X)+3A(X,Y).$$
Taking into account that the scalar curvature is constant and $T_{00} = 0$ the above  equation takes the following explicit form
\begin{multline}\label{e:NablaTaib}
8\nabla T^0(Y,X, I_i \nabla f) = -12\nabla T^0(\nabla f, I_iX,Y) -12 \nabla T^0(\nabla f,X, I_iY)  +48g(X,Y) f_i\\
 - \sum_{s=1}^3 f_s[30g(I_iX, I_sY) + 18g(I_s X, I_iY)] +12 \sum_{s=1}^3 \nabla^2 f(\xi_i, \xi_s)g(X,I_sY)\\
 + 8f T^0(X, I_iY)  +12fg(X, I_i Y)+ f_i[24T^0(I_iX, I_iY) - 32 T^0(X,Y)]\\
   + f_j[17T^0(X, I_kY) + 15T^0(I_jX, I_iY) + 9T^0(I_iX, I_jY) + 15 T^0(I_kX,Y)] \\
    + f_k [15T^0(I_kX, I_iY) + 9 T^0(I_iX, I_kY) - 17T^0(X, I_jY) - 15 T^0(I_j X, Y)]
    -  48\Gamma^i(X,Y).
\end{multline}

Now, let $X=Y = I_j\nabla f$ in \eqref{e:NablaTaib} and use \eqref{e:T0i} to obtain
\[
8T_{ji:j} = -24 T_{jk;0} +8fT_{jk} + f_i[24T_{kk} - 32T_{jj}] - 32f_j T_{ij} - 24f_k T_{ik} - 48\Gamma^i(I_j \nabla f, I_j \nabla f).
\]

Next, consider \eqref{e:NablaTaib} written for $j$, and then let $X = I_i\nabla f$, $Y = I_j\nabla f$. Using \eqref{Twhens=2}and  $T_{i0;0} = 0$ by \eqref{e:T0i00}  it follows
\begin{multline*}
8T_{ij;j} = 12T_{kj;0} +12|\nabla f|^2 f_i - 12|\nabla f|^2\nabla^2 f(\xi_j, \xi_k) - 32 f_j T_{ij} + 26f_k T_{ik} + f_i[-9T_{kk} + 17 T_{ii} - 15 T_{jj}] \\
- 48\Gamma^j(I_i \nabla f, I_j \nabla f).
\end{multline*}
By the symmetry of $T_{\alpha\beta;\gamma}$ in its first two indices and the above identities for $T_{ji;j}$ and $T_{ij;j}$ we have
\begin{multline}\label{e:l1}
0 = 8T_{ji;j} - 8T_{ij;j} =-36T_{jk;0}  + 8fT_{jk} + 33 f_i T_{kk} -17f_i T_{jj}  - 50 f_k T_{ki} - 17 f_i T_{ii}
- 12|\nabla f|^2 f_i \\
+ 12|\nabla f|^2 \nabla^2 f(\xi_j, \xi_k)  - 48[\Gamma^i(I_j\nabla f, I_j \nabla f) - \Gamma^j(I_i \nabla f, I_j \nabla f)].
\end{multline}
Since $S=2$, \eqref{e:divT0} and \eqref{e:divT0i2} imply $\nabla^*T^0 = 0$. Therefore by the last lines in \eqref{e:2.12}, \eqref{e:2.10} we  have
\begin{equation}\label{e:rhoS}
\rho_k(I_j X, \xi_i) = - \rho_k(I_i X, \xi_j) =  \frac{1}{12}\nabla_k^*T^0(I_k X).
\end{equation}
The definition of $\Gamma^i(Y,X)$ in \eqref{e:gamma}, together with \eqref{e:rhoS} and \eqref{e:f2} show that
\[
\Gamma^i(I_j \nabla f, I_j \nabla f) - \Gamma^j(I_i \nabla f, I_j \nabla f) = -\frac{1}{4}|\nabla f|^2 \nabla_k^*T^0(I_j\nabla f) =  fT_{jk} + f_j T_{ij} - f_i T_{jj},
\]
which gives a formula for the last term in \eqref{e:l1}. The latter, together with  the identities \eqref{e:Tij0} for the term $T_{jk;0} $  and \eqref{e:vHessfjkS} for the term $|\nabla f|^2 \nabla^2 f(\xi_j, \xi_k)$, allows to rewrite \eqref{e:l1} as follows
\begin{multline}
0 =-9[ f_i(T_{jj} - T_{kk}) + f_k T_{ik} - f_j T_{ij}] + 8fT_{jk} + 33 f_i T_{kk} -17f_i T_{jj}  - 50 f_k T_{ki} - 17 f_i T_{ii}
\\- 12|\nabla f|^2 f_i
 -8fT_{jk} + 12 |\nabla f|^2 f_i - 8 f_j T_{ij} + 8 f_i T_{jj}- 48[ fT_{jk} + f_j T_{ij} - f_i T_{jj}] \\ = 30f_i T_{jj} + 42 f_i T_{kk} - 17 f_i T_{ii} -47f_j T_{ij} - 59f_k T_{ik} - 48f T_{jk} .
\end{multline}
From \eqref{Twhens=2} we have that $-17f_i T_{ii} =17 f_j T_{ij} + 17 f_k T_{ik} $, therefore the above reads
\begin{equation}\label{m1}
0   = 30 f_i T_{jj} +42 f_i T_{kk} - 30 f_j T_{ij} - 42 f_k T_{ik} -  48f T_{jk}.
\end{equation}
In addition, \eqref{Twhens=2} also gives
\begin{equation}\label{m2}
f_j T_{ij} - f_i T_{jj} = f_j T_{ij} + f_i T_{ii} + f_i T_{kk} = f_i T_{kk} - f_i T_{ik}.
\end{equation}
Applying this to \eqref{m1} shows
\[
0 = 30 f_i T_{jj} +42 f_i T_{kk} - 30 f_j T_{ij} - 42 f_k T_{ik} -  48f T_{jk} = 12 f_i T_{kk} - 12 f_k T_{ik}  - 48 fT_{jk}
\]
from which the first identity in \eqref{e:fTjk} follows.

We turn to the proof of \eqref{e:fTii}. Choosing $X$ and $Y$ in the obvious ways, equation \eqref{e:NablaTaib} written for $j$ and $k$, respectively, implies the following identities
\begin{multline*}
8T_{kj;i} = 12T_{kk;0} - 12T_{ii;0} - 12|\nabla f|^2 \nabla^2 f(\xi_j, \xi_j)- 8fT_{kk}  -12|\nabla f|^2 f
 - 56f_j T_{ki} - 6 f_k T_{ij} - 2 f_i T_{jk}\\ - 48\Gamma^j(I_k\nabla f, I_i \nabla f)
\end{multline*}
and
\begin{multline*}
8T_{jk;i} = 12T_{ii;0}- 12 T_{jj;0} + 12|\nabla f|^2 \nabla^2 f(\xi_k, \xi_k) +8fT_{jj} + 12|\nabla f|^2 f
- 56 f_k T_{ij} - 2 f_i T_{jk} - 6 f_j T_{ki} \\
- 48\Gamma^k(I_j\nabla f, I_i \nabla f) .
\end{multline*}
Therefore, we have
\begin{multline}\label{e:l2}
0 = 8T_{kj;i} - 8T_{jk;i} = 12[T_{kk;0} - 2T_{ii;0} + T_{jj:0}] - 12|\nabla f|^2[\nabla^2 f(\xi_j, \xi_j) + \nabla^2 f(\xi_k,\xi_k)] \\
- 24|\nabla f|^2 f - 8f[T_{kk} + T_{jj}] - 50[f_j T_{ki} -f_k T_{ij}] -48[\Gamma^j(I_k\nabla f, I_i\nabla f) - \Gamma^k(I_j\nabla f, I_i\nabla f)].
\end{multline}
By the definition of $\Gamma^i$ in \eqref{e:gamma},  followed by the identity \eqref{e:rhoS} for $\rho_s$, and \eqref{e:f1}, we find
\begin{multline}
\Gamma^j(I_k\nabla f, I_i \nabla f) - \Gamma^k(I_j\nabla f, I_i \nabla f)  = \frac{1}{12}|\nabla f|^2 \nabla_k^*T^0(I_k\nabla f) \\-\frac{1}{6}|\nabla f|^2 \nabla_i^*T^0(I_i\nabla f) + \frac{1}{12}|\nabla f|^2 \nabla_j^*T^0(I_j\nabla f) = fT_{ii} - f_j T_{ki} + f_k T_{ij}.
\end{multline}
\noindent Then, using the above along with \eqref{e:Tii0} and \eqref{e:vHessfii} in \eqref{e:l2} gives
\begin{multline}
0 = 12\left[\frac{1}{2}(f_i T_{jk} - f_j T_{ij}) - (f_j T_{ki} - f_k T_{ji}) + \frac{1}{2}(f_k T_{ij} - f_i T_{kj})\right] \\-12\left[-|\nabla f|^2 f - \frac{2}{3}(fT_{jj} + f_i T_{jk} - f_k T_{ij}) - |\nabla f|^2 f - \frac{2}{3}(f T_{kk} + f_j T_{ki} - f_i T_{jk})\right] - 24|\nabla f|^2 f \\ - 8f[T_{kk} + T_{jj}] - 50[ f_j T_{ki} - f_k T_{ij}] - 48[fT_{ii} - f_j T_{ki} + f_k T_{ij}] \\ = -12f_j T_{ki} + 12 f_k T_{ij} - 48 f T_{ii}
\end{multline}
from which \eqref{e:fTii} follows.
\end{proof}

\subsection{Proof of Theorem \ref{t:mainpan} } With the notation set in \eqref{sigalph} we have $T_{00} = T_{0i} =0$, see \eqref{Twhens=2} and \eqref{e:T0i}, hence
\begin{equation}\label{normT0}
|\nabla f|^4|T^0|^2 = T_{11}^2 + T_{22}^2 + T_{33}^2 + 2T_{12}^2 + 2T_{23}^2 + 2T_{31}^2
=\sum_{(i\,j\,k)} [T_{ii}^2 + 2T_{ij}^2]=\sum_{(i\,j\,k)} [T_{ii}^2 + 2T_{jk}^2],
 \end{equation}
 recalling that $\sum_{(i\,j\,k)}$ indicates a cyclic sum. Using the identities $4fT_{jk} = f_i T_{kk} - f_k T_{ki}=f_j T_{ij} - f_i T_{jj}$ and $4fT_{ii} = f_k T_{ij} - f_j T_{ki}$ by \eqref{e:fTjk} and \eqref{e:fTii}, we obtain
\begin{multline}
4f|\nabla f|^4|T^0|^2 =\sum_{(i\,j\,k)} \left[T_{ii}\left( f_kT_{ij}-f_jT_{ki}\right) + T_{jk}\left( f_i T_{kk} - f_k T_{ki}\right)  + T_{jk}\left(f_j T_{ij} - f_i T_{jj}\right) \right]\\
=\sum_{(i\,j\,k)} \left[  f_kT_{ii}T_{ij}-f_jT_{ii}T_{ki} +  f_i T_{jk}T_{kk} - f_kT_{jk} T_{ki}  + f_j T_{jk}T_{ij} - f_i T_{jk}T_{jj} \right]\\
=\sum_{(i\,j\,k)} \left[  f_kT_{ii}T_{ij}-f_kT_{jj}T_{ij} +  f_k T_{ij}T_{jj} - f_kT_{jk} T_{ki}  + f_k T_{ki}T_{jk} - f_k T_{ij}T_{ii} \right]=0.
\end{multline}
By  Lemma \ref{elliplem}  it follows $T^0 \equiv 0$. Thus,  $M$ is a qc-Einstein structure. The conclusion that $(M,\eta)$ is qc-equivalent to the standard $3$-Sasakian sphere then follows from the second part of \cite[Theorem $8.3$]{IV15}.

\end{document}